\documentclass{amsart}
\usepackage[utf8]{inputenc}
\usepackage{amssymb,amsmath,amsthm}
\usepackage{mathtools}
\usepackage[all,cmtip,2cell]{xy} 
\usepackage{verbatim}
\usepackage{enumitem}
\usepackage[unicode,psdextra,draft=false]{hyperref}
\usepackage[capitalise]{cleveref}
\usepackage{xcolor}
\usepackage{xspace}
\usepackage[left=3cm,right=3cm,top=3cm,bottom=4.5cm]{geometry}
\usepackage{tikz-cd}
\usepackage{url}
\usepackage{enumitem}

\newcommand{\id}{\operatorname{id}}
\DeclareMathOperator*{\colim}{colim}

\newcommand{\catC}{\mathcal{C}}
\newcommand{\catD}{\mathcal{D}}

\newcommand{\cSet}{\mathsf{cSet}}
\newcommand{\Set}{\mathsf{Set}}
\newcommand{\sSet}{\mathsf{sSet}}
\newcommand{\FP}{\mathcal{P}}
\newcommand{\FS}{\mathcal{S}}
\newcommand{\zvec}{\overrightarrow{0}}
\newcommand{\ovec}{\overrightarrow{1}}
\newcommand{\avec}{\overrightarrow{a}}
\newcommand{\pushout}{\arrow [dr, phantom, "\ulcorner" very near end]}
\newcommand{\pullback}{\arrow [dr, phantom, "\lrcorner" very near start]}

\newcommand{\zo}{\{0,1\}}

\newcommand{\ihom}{\underline{\operatorname{hom}}}

\newcommand{\bd}{\partial}

\newcommand{\adjoint}{\dashv}

\newcommand{\conset}{\{\wedge,\vee\}}

\newcommand{\hcap}{\widehat{\sqcap}}

\newcommand{\hBox}{\widehat{\Box}}

\newcommand{\Poset}{\mathsf{Poset}}

\theoremstyle{plain}
	\newtheorem{thm}{Theorem}[section]
	\newtheorem*{thm*}{Theorem}
	\newtheorem{cor}[thm]{Corollary}
	\newtheorem*{cor*}{Corollary}
	\newtheorem{prop}[thm]{Proposition}
	\newtheorem*{prop*}{Proposition}
	\newtheorem{lem}[thm]{Lemma}
	\newtheorem*{lem*}{Lemma}
\theoremstyle{definition}
	\newtheorem{Def}[thm]{Definition}
	\newtheorem*{Def*}{Definition}
	\newtheorem{rmk}[thm]{Remark}
	\newtheorem*{rmk*}{Remark}
    
    \newtheorem{examples}[thm]{Examples}

\Crefname{lem}{Lemma}{Lemmas}
\Crefname{prop}{Proposition}{Propositions}

\title{Symmetry in the cubical Joyal model structure}
\author{Brandon Doherty}
\date{}

\begin{document}

\begin{abstract}
 We study properties of the cubical Joyal model structures on cubical sets by means of a combinatorial construction which allows for convenient comparisons between categories of cubical sets with and without symmetries. In particular, we prove that the cubical Joyal model structures on categories of cubical sets with connections are cartesian monoidal. Our techniques also allow us to prove that the geometric product of cubical sets (with or without connections) is symmetric up to natural weak equivalence in the cubical Joyal model structure, and to obtain induced model structures for $(\infty,1)$-categories on cubical sets with symmetries.
\end{abstract}

\maketitle

\section{Introduction}

Model structures on categories of \emph{cubical sets} can be used to model higher categories. Much like the more familiar simplicial sets, cubical sets are presheaves on a small indexing category (a ``cube category'') which can be thought of as collections of cubes in all dimensions. Unlike in the case of simplicial sets, there are many possible choices of this indexing category, giving rise to categories of cubical sets whose cubes are related by various different kinds of structure maps.

In particular, models of $\infty$-groupoids (the \emph{Grothendieck model structures} \cite{CisinskiAsterisque,CisinskiUniverses}) and $(\infty,1)$-categories (the \emph{cubical Joyal model structures} \cite{doherty-kapulkin-lindsey-sattler}), Quillen equivalent to the standard simplicial models, have been established using \emph{minimal cubical sets} and \emph{cubical sets with connections}. These kinds of cubical sets have a relatively simple set of structure maps: \emph{face maps}, \emph{projections} which play the role of degeneracies, and in the latter case, an additional kind of degeneracy called \emph{connections}. Other cube categories include \emph{symmetries}, automorphisms corresponding to the natural symmetries of cubical shapes.

Pre-composition with an inclusion of cube categories $i \colon \Box_A \hookrightarrow \Box_B$ defines a forgetful functor between the corresponding categories of cubical sets (which ``forgets the additional structure maps of $\Box_B$''), having both a left and a right adjoint. The aim of this paper is to study the cubical Joyal model structures by means of such comparison functors, specifically between the categories of cubical sets described above and those with symmetries. In doing so, we will also prove or re-prove analogous results involving the Grothendieck model structures.

Our main results involve the \emph{geometric product} of cubical sets, a monoidal product $\otimes$ having the useful property that a geometric product of cubes is again a cube; this is the most convenient monoidal product for parameterizing homotopy in the cubical Joyal model structures. One drawback of the geometric product, at least for the kinds of cubical sets on which cubical Joyal model structures have been established, is that it is not symmetric -- in general, $X \otimes Y$ is not isomorphic to $Y \otimes X$. Our first main result establishes that the geometric product is ``symmetric up to natural weak equivalence'', for both minimal cubical sets and those with connections.

\begin{thm*}[cf \cref{monoidal-weak-equiv}]
    For cubical sets $X$ and $Y$, the geometric products $X \otimes Y$ and $Y \otimes X$ are naturally weakly equivalent.
\end{thm*}

Our results also allow us to obtain alternative models of $(\infty,1)$-categories via model structures on categories of cubical sets with symmetries.

\begin{thm*}[cf \cref{induced-model-structures}]
    Cubical sets with symmetries and connections admit model structures for $(\infty,1)$-categories left- and right-induced by the forgetful functor to minimal cubical sets, and Quillen equivalent to the cubical Joyal model structures.
\end{thm*}

We also study the cartesian product of cubical sets and its relationship with the geometric product. In the absence of connections, the cartesian product is not homotopically well-behaved (see \cite[Rmk.~3.5]{jardine:categorical-homotopy-theory} for instance). On the other hand, cubical sets with connections form a strict test category, implying compatibility of the cartesian product with the Grothendieck model structure. Our next main result generalizes this observation to the setting of $(\infty,1)$-categories.

\begin{thm*}[cf \cref{X-Y-comparison-weq,cartesian-monoidal}]
    The cubical Joyal model structure on cubical sets with connections is monoidal with respect to the cartesian product. Moreover, the geometric and cartesian products are naturally weakly equivalent in the cubical Joyal model structure.
\end{thm*}

The key combinatorial construction underlying all of our results is that of the \emph{standard decomposition cubes}, which allow for certain inclusions of cubical sets to be expressed as transfinite composites of open-box fillings. Though similar in spirit to the open-box filling techniques used in \cite{doherty:without-connections} to analyze the relationship between minimal cubical sets and those with connections, this construction does not rely on intricate analysis of standard factorizations into generators of maps in cube categories, as the techniques of that paper did. Instead, it can be described by simple formulas (see \cref{N-def,N-explicit}), and admits a more intuitive conceptual understanding (see \cref{N-examples}). 

\subsection*{Related work}

Model structures for $\infty$-groupoids on categories of cubical sets with symmetries have been the subject of considerable study, much of it motivated by applications to type theory; see, for instance, \cite{isaacson:symmetric,awodey:cartesian-cubical-model-categories,cavallo-sattler:relative-elegance,awodey-cavallo-coquand-riehl-sattler}. Induced model structures on one particular category of cubical sets with symmetries, for both $\infty$-groupoids and $(\infty,1)$-categories, have  been constructed in \cite{hackney-rovelli:left-induced}; the model structures in that reference are induced by a functor from cubical to simplicial sets, whereas ours are induced by forgetful functors from symmetric cubical sets to those without symmetries.

The topic of the compatibility of cartesian products of cubical sets with model structures for $\infty$-groupoids has previously been studied in terms of the theory of strict test categories; see in particular \cite{maltsiniotis:connections-strict-test-cat}, in which it is shown that the cube category with faces, projections and connections is a strict test category, and \cite{buchholtz-morehouse:varieties-of-cubes}, in which this result is extended to cube categories with additional structure maps.

\subsection*{Acknowledgements}

While working on this paper, the author was supported in part by a grant from the Knut and Alice Wallenberg Foundation, entitled ``Type Theory for Mathematics and Computer Science'' (principal investigator: Thierry Coquand).

The author would like to thank Timothy Campion, Evan Cavallo, Krzysztof Kapulkin, and Peter LeFanu Lumsdaine for helpful discussions on the topic of this paper. The author would also like to thank the anonymous referee for numerous helpful comments, in particular for suggesting the current proof structure of \cref{bdry-colim}, for identifying an error affecting the original proofs of \cref{i-shriek-pairs,N-active}, and for the observation about test model structures which underlies \cref{left-induced-test} and the current form of \cref{HR-comparison}.

\subsection*{Organization of the paper}

In \cref{sec background}, we define the categories of cubical sets under consideration and discuss some of the basic constructions used in their study, including the geometric product. Most of this is background material, but we will also define some concepts of specific relevance to the present work and prove some basic combinatorial lemmas which will be of use in later sections. In \cref{sec weak sym}, we introduce the adjoint triples $i_! \adjoint i^* \adjoint i_*$ which relate different categories of cubical sets, as well as the standard decomposition cubes which allow the units of the adjunctions $i_! \adjoint i^*$ to be analyzed by open-box filling. In particular, \cref{N-closed-anodyne} is our key technical result, of which all of this paper's main theorems are consequences. In \cref{sec unit} we apply \cref{N-closed-anodyne} to show that, given certain conditions on the categories of cubical sets under consideration, the unit of the comparison adjunction $i_! \adjoint i^*$ is a natural trivial cofibration in the cubical Joyal model structure. Using this, we show that the geometric product is symmetric up to natural weak equivalence, and construct induced model structures for $(\infty,1)$-categories and $\infty$-groupoids on cubical sets with symmetries. In \cref{sec monoidal}, we further apply \cref{N-closed-anodyne} to show that the cubical Joyal model structure on cubical sets with connections is cartesian monoidal, and is naturally weakly equivalent to the geometric product.

\section{Background} \label{sec background}

\subsection{Categories of cubical sets}

We begin by introducing the categories which will be our primary objects of study.

\begin{Def}\label{cube-cat}
The \emph{full cube category} $\Box_\FS$ is the full subcategory of $\Set$ on the objects $\{0,1\}^n, n \geq 0$.
This category is generated under composition by the following classes of maps (where $\wedge$ and $\vee$ respectively denote the minimum and maximum functions $\zo^2 \to \zo$):
\begin{itemize}
  \item \emph{faces} $\partial^n_{i,\varepsilon} \colon \zo^{n-1} \to \zo^n$ for $i = 1, \ldots , n$ and $\varepsilon = 0, 1$ given by:
  \[ \partial^n_{i,\varepsilon} (a_1, a_2, \ldots, a_{n-1}) = (a_1, a_2, \ldots, a_{i-1}, \varepsilon, a_i, \ldots, a_{n-1})\text{;}  \]
  \item \emph{diagonals} $\delta^n_{i} \colon \zo^{n-1} \to \zo^n$ for $i = 1, \ldots, n-1$, given by:
    \[ \delta^n_{i} (a_1, a_2, \ldots, a_{n-1}) = (a_1, a_2, \ldots, a_{i-1}, a_i, a_i, a_{i+1}, \ldots, a_{n-1})\text{;}  \]
  \item \emph{projections} $\sigma^n_i \colon \zo^{n+1} \to \zo^{n}$ for $i = 1, 2, \ldots, n + 1$ given by:
  \[ \sigma^n_i ( a_1, a_2, \ldots, a_{n+1}) = (a_1, a_2, \ldots, a_{i-1}, a_{i+1}, \ldots, a_{n+1})\text{;}  \]
   \item \emph{positive connections} $\gamma^n_{i,1} \colon \zo^{n+1} \to \zo^{n}$ for $i = 1, 2, \ldots, n$ given by:
  \[ \gamma^n_{i,1} (a_1, a_2, \ldots, a_{n+1}) = (a_1, a_2, \ldots, a_{i-1},  a_i \wedge a_{i+1}, a_{i+2}, \ldots, a_{n+1}) \text{.} \]
  \item \emph{negative connections} $\gamma^n_{i,0} \colon \zo^{n+1} \to \zo^{n}$ for $i = 1, 2, \ldots, n$ given by:
  \[ \gamma^n_{i,0} (a_1, a_2, \ldots, a_{n+1}) = (a_1, a_2, \ldots, a_{i-1},  a_i \vee a_{i+1}, a_{i+2}, \ldots, a_{n+1}) \text{.} \]
  \item \emph{transpositions} $\lambda^n_i \colon \zo^n \to \zo^n$ for $1 \leq i \leq n - 1$ given by
    \[ \lambda^n_{i} (a_1, a_2, \ldots, a_n) = (a_1, a_2, \ldots, a_{i-1}, a_{i+1}, a_{i}, a_{i+2}, \ldots, a_n) \text{.} \]
      \item \emph{reversals} $\rho^n_i \colon \zo^n \to \zo^n$ for $1 \leq i \leq n$ given by
    \[ \rho^n_{i} (a_1, a_2, \ldots, a_n) = (a_1, a_2, \ldots, a_{i-1}, 1-a_{i}, a_{i+1}, \ldots, a_n) \text{.} \]
\end{itemize}
\end{Def}

For clarity, the superscript $n$ will typically be omitted when it is irrelevant or clear from context.

Analogously to the familiar simplicial identities, one could describe the composition of these maps in terms of \emph{cubical identities}, as is done, for instance, in \cite{grandis-mauri} for the faces, projections and connections. As this would be both unwieldy and unnecessary for the present work, we will not produce a full list of cubical identities here; instead, we simply record a few identities which will be of particular use in our calculations, all of which can easily be derived from the definitions above.

\begin{lem}\label{some-cubical-identities}
    The maps above satisfy the following identities:
    \begin{itemize}
        \item $\sigma_i \bd_{i,\varepsilon} = \id$;
        \item $\gamma_{i,\varepsilon} \bd_{i,\varepsilon} = \gamma_{i,\varepsilon} \bd_{i+1,\varepsilon} = \id$;
        \item $\gamma_{i,\varepsilon} \bd_{i,1-\varepsilon} = \gamma_{i,\varepsilon} \bd_{i+1,1-\varepsilon} = \bd_{i,1-\varepsilon} \sigma_i$. \qed
    \end{itemize}
\end{lem}

A \emph{cube category} is a subcategory of $\Box_\FS$ generated by some subset of the generating classes above which includes the faces and projections. We denote cube categories by $\Box_A$, where $A$ is a (possibly empty) subset of the set of symbols $\{\wedge, \vee, \Sigma, \rho, \delta\}$. For such a subset, $\Box_A$ denotes the cube category generated by faces and projections, as well as:
\begin{itemize}
\item positive connections if $\wedge \in A$;
\item negative connections if $\vee \in A$;
\item transpositions if $\Sigma \in A$;
\item reversals if $\rho \in A$;
\item diagonals if $\delta \in A$.
\end{itemize}

By composing transpositions, we obtain all automorphisms of the posets $\{0 \leq 1\}^n$, which act by permutation of coordinates. We will refer to these maps as \emph{symmetries}, and speak of a cube category ``containing symmetries'' rather than ``containing transpositions''. 

In this context we let $\FS$ denote the entire set $\{\wedge, \vee, \Sigma, \rho, \delta\}$, so that an arbitrary cube category can be written as $\Box_A$ for $A \subseteq \FS$. For ease of notation, when writing subsets of $\FS$ in subscripts we will not use braces or commas, so that, for instance, the cube category with both kinds of connections is denoted $\Box_{\wedge \vee}$. When viewing $\zo^n$ as an object of $\Box_A$, we will denote it by $\Box^n_A$. 

Note that the positive connections and reversals together generate the negative connections, and vice-versa; thus we will implicitly assume, when dealing with an arbitrary subset $A \subseteq \FS$, that if $A$ contains both $\rho$ and $\wedge$ (resp.~$\vee$) then $A$ contains $\vee$ (resp.~$\wedge$) as well.

The category of presheaves on $\Box_A$, ie~contravariant functors from $\Box_A$ to $\Set$, will be denoted $\cSet_A$. Objects in such a category are called \emph{cubical sets}, and may be viewed as collections of cubes in each dimension; for $X \in \cSet$, the image under $X$ of $\Box^n_A$ is denoted $X_n$, and thought of as the set of $n$-cubes of $X$. By a standard abuse of notation, we write the representable presheaf $\Box_A(-,\Box^n_A)$ as $\Box^n_A$. We refer to this object as the \emph{$n$-cube} in $\cSet_A$; this may be viewed as consisting of a single cube of dimension $n$ and all of its faces. 

\begin{examples}
We list specific examples of notable cube categories.
\begin{itemize}
\item $\Box_\varnothing$ is the \emph{minimal cube category}, having only faces and projections. The corresponding category of cubical sets $\cSet_\varnothing$ was originally studied by Kan in \cite{kan:abstract-htpy-1}, and more recently by Cisinski in \cite{CisinskiAsterisque}.
\item $\Box_{\wedge \vee \Sigma \delta}$ is the subcategory of maps which respect the partial ordering on $\{0,1\}^n$ induced by the ordering on $\{0,1\}$ given by $0 \leq 1$; in other words, it is the full subcategory of $\Poset$ on the powers of the interval $[1] = \{0 \leq 1\}$. As this cube category is of particular interest, we will denote it $\Box_\FP$ for ease of reference (so $\FP = \{\wedge, \vee, \Sigma, \delta\}$).
\end{itemize}
\end{examples}

Before continuing, we should comment on certain ambiguities of terminology. In many sources, in particular \cite{doherty-kapulkin-lindsey-sattler}, the term ``degeneracy'' is used to refer specifically to the projections. Our terminology instead follows \cite{campion:EZ-cubes}, as it is convenient to reserve ``degeneracy'' for a catch-all term for the strictly degree-decreasing maps in the Eilenberg-Zilber structures to be described in \cref{sec combinatorics}. Likewise, in many sources, the term ``diagonal'' refers to the maps $\zo^n \to \zo^{2n}$ which repeat an entire string of coordinates, ie, those which send $(a_1,\ldots,a_n)$ to $(a_1,\ldots,a_n,a_1,\ldots,a_n)$; the well-known observation that diagonals and projections together generate the symmetries refers to these diagonals, not to those specified in our list above. Our usage instead follows \cite{buchholtz-morehouse:varieties-of-cubes}, as we wish to regard the diagonals and symmetries as distinct generating classes. 

We will denote the $n$-tuple $(0,\ldots,0)$ by $\zvec_n$; likewise $(1,\ldots,1)$ will be denoted $\ovec_n$. When there is no risk of confusion we will omit the subscripts and simply write $\zvec, \ovec$.

We will occasionally represent cubical sets visually.
For a $1$-cube $f$, we draw
\[
\xymatrix{ x \ar[r]^f & y} \]
to indicate $x = f \partial_{1,0}$ and $y = f \partial_{1,1}$.
For a $2$-cube $s$, we draw
\[
\xymatrix{
 x
  \ar[r]^h
  \ar[d]_f
&
   y
  \ar[d]^g
\\
  z
  \ar[r]^k
&
 w
}
\]
to indicate $s\partial_{1,0} = f$,  $s\partial_{1,1} = g$, $s\partial_{2,0} = h$,  and $s\partial_{2,1} = k$.
As for the convention when drawing $3$-dimensional cubes, we use the following ordering of axes: 
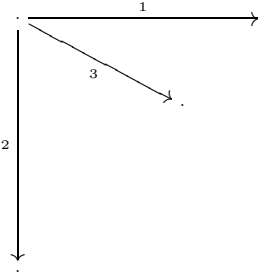
\begin{tiny}
\[
\xymatrix{
 \cdot
  \ar[rrr]^1
  \ar[rrd]_3
  \ar[ddd]_2
& & &
   \cdot
\\
  & & \cdot
\\ & & &
\\
 \cdot & & &
}
\]
\end{tiny}
For readability, we do not label $2$- and $3$-cubes.

Lastly, a degenerate $1$-cube $x \sigma_1$ on $x$ is represented by
\[
\xymatrix{ x \ar@{=}[r] & x\text{,}} \]
while a $2$- or $3$-cube whose boundary agrees with that of a degenerate cube is assumed to be degenerate unless indicated otherwise.
For instance, a $2$-cube depicted as
\[
\xymatrix{
 x
  \ar@{=}[r]
  \ar[d]_f
&
   x
  \ar[d]^f
\\
  y
  \ar@{=}[r]
&
 y
}
\]
represents $f \sigma_1$.

\subsection{The geometric product}

The cartesian product on $\Set$ restricts to a monoidal product $\otimes$ on each category $\Box_A$. The faithfulness of the cartesian product implies the following useful result.

\begin{lem}\label{tensor-faithful}
For any $A \subseteq \FS$, the functor $\otimes \colon \Box_A \times \Box_A \to \Box_A$ is faithful. \qed
\end{lem}

We may extend this to a monoidal product on $\cSet_A$, called the \emph{geometric product}, by Day convolution. In other words, we define the geometric product $\otimes \colon \cSet_A \times \cSet_A \to \cSet_A$ by left Kan extension, as depicted below:
\[
\xymatrix@C+0.5cm{
  \Box_A \times \Box_A \ar[r]^{\otimes} \ar[d] & \Box_A \ar[r] 
&
  \cSet_A
\\
  \cSet_A \times \cSet_A
  \ar[rru]_{\otimes}
&
}
\]
The unit of this monoidal structure is the 0-cube $\Box^0_A$.

From this construction we obtain a formula for the geometric product $X \otimes Y$ as a colimit over the cubes of $X$ and $Y$:
\[
X \otimes Y = \colim\limits_{\substack{x \colon \Box_A^m \to X \\ y \colon \Box_A^n \to Y}} \Box_A^{m+n}
\]

Given a pair of cubes $x \colon \Box^m_A \to X, y \colon \Box^n_A \to Y$, the component of the colimit cone corresponding to the pair $(x,y)$ is a cube $\Box^{m+n}_A \to X \otimes Y$, which we denote by $x \otimes y$. As the notation suggests, this is also the image of $(x,y)$ under the bifunctor $\otimes \colon \cSet_A \times \cSet_A \to \cSet_A$. The colimit formula above allows us to obtain the following characterization of $X \otimes Y$, generalizing \cite[Prop.~1.24]{doherty-kapulkin-lindsey-sattler}.

\begin{prop} \label{geo-prod-characterization}
    For $X, Y \in \cSet_A$, every cube of $X \otimes Y$ is of the form $(x \otimes y) \phi$ for some $x \colon \Box^m_A \to X$, some $y \colon \Box^n_A \to Y$, and some $\phi \colon \Box^k_A \to \Box^{m+n}_A$. Moreover, these cubes are subject to the following identification: for $x, y$ as above and $\psi \colon \Box^{m'}_A \to \Box^m, \psi' \colon \Box^{n'}_A \to \Box^n_A$, we have $(x \otimes y) (\psi \otimes \psi') = (x \psi) \otimes (y \psi')$. \qed
\end{prop}

Note that for most cube categories $\Box_A$, the geometric product of cubical sets does not coincide with the cartesian product, and is not even a symmetric monoidal product. However, it is symmetric if $\Sigma \in A$, and coincides with the cartesian product if $\Sigma, \delta \in A$. (These results are well-known, but we will re-prove the latter as \cref{geometric-cartesian-iso}.)

In cases where the two do not coincide, the geometric product is better behaved than the cartesian product -- for instance, in every cube category $\Box_A$, for $m, n \geq 0$ we have $\Box^m_A \otimes \Box_A^n = \Box_A^{m+n}$, while the analogous isomorphism generally does not hold for the cartesian product. Moreover, in the absence of connections, a cartesian product of cubes may not even have the homotopy type of a contractible space (see \cite[Rmk.~3.5]{jardine:categorical-homotopy-theory}).

Given a cubical set $X$, the left tensor $- \otimes X$ and the right tensor $X \otimes -$ define two functors $\cSet_A \to \cSet_A$; as discussed above, these are not isomorphic unless $\Box_A$ contains symmetries. Both of these functors admit right adjoints, and we write $\ihom_L(X, -)$ for the right adjoint of the left tensor and $\ihom_R(X, -)$ for the right adjoint of the right tensor. (Of course, in the symmetric case these are naturally isomorphic.)  Explicitly, these functors are given by $\ihom_L(X,Y)_{n} = \cSet(\Box^n \otimes X,Y)$, $ \ihom_R(X,Y)_{n} = \cSet(X \otimes \Box^n, Y)$.
Thus the monoidal structure on $\cSet_A$ given by the geometric product is closed.

\subsection{Combinatorics of cubical sets} \label{sec combinatorics}

We next discuss the convenient combinatorial properties of certain categories of cubical sets, which are of use in studying their homotopy theory.

\begin{prop}[{\cite[Cor.~7.9]{campion:EZ-cubes}}]\label{cube-EZ}
If $\delta \notin A$, then $\Box_A$ is an Eilenberg-Zilber category, with the following structure:
\begin{itemize}
\item for all $n$, $\deg([1]^n) = n$;
\item $(\Box_A)_+$ is generated under composition by the faces, symmetries, and reversals;
\item $(\Box_A)_-$ is generated under composition by the projections, connections, symmetries, and reversals. \qed
\end{itemize}
\end{prop}

In keeping with standard usage in the theory of Eilenberg-Zilber categories, we will refer to the non-isomorphisms of $(\Box_A)_-$ as \emph{degeneracies}. Given a cubical set $X$ and an $m$-cube $x \in X_m$, we say $x$ is \emph{degenerate} if it is in the image of the structure map $X_n \to X_m$ induced by some degeneracy $\Box^m_A \to \Box^n_A$, and \emph{non-degenerate} otherwise.

\begin{lem}[{\cite[Prop.~1.18]{doherty-kapulkin-lindsey-sattler}}]\label{EZ-for-cubes}
Let $X \in \cSet_A$ for $A \subseteq \conset$. Then each cube $x \colon \Box^n_A \to X$ can be written as $y \phi$ for a unique non-degenerate cube $y \colon \Box^m_A \to X$ and a unique map $\phi \colon \Box^n_A \to \Box^m_A$ in $(\Box_A)_-$. \qed
\end{lem}

In cubical sets with symmetries or reversals, but without diagonals, as well as those with symmetries and diagonals but without connections, an analogue of \cref{EZ-for-cubes} holds up to isomorphism; see \cite[Thms.~7.9 \& 8.12]{campion:EZ-cubes}. In our present work, however, we will have no need of this stronger result.

The following consequence of \cref{EZ-for-cubes} will be of use in recognizing non-degenerate cubes.

\begin{cor}\label{degen-one-face}
Let $X \in \cSet_A$ with $A \subseteq \conset$ and $x \colon \Box_A^m \to X, y \colon \Box^{m-1}_A \to X$, with $x$ degenerate and $y$ non-degenerate. Then it is not the case that $x \bd_{i,\varepsilon} = y$ for exactly one face map $\bd_{i,\varepsilon}$. 
\end{cor}

\begin{proof}
We note that by \cref{EZ-for-cubes}, we have $x = z \phi$ for a unique degeneracy $\phi \colon \Box^m_A \to \Box^n_A$ and a unique non-degenerate $n$-cube $z \colon \Box^n_A \to X$. In the case $n < m - 1$, consider a face $x \bd_{i,\varepsilon} = z \phi \bd_{i,\varepsilon}$; our assumption on $m$ implies that $\phi \bd_{i,\varepsilon} \colon \Box^{m-1}_A \to \Box^n_A$ strictly decreases degree. By \cref{cube-EZ}, we can factor $\phi \bd_{i,\varepsilon}$ as $\kappa \phi'$ for some $\kappa$ in $(\Box_A)_+$ and $\phi'$ in $(\Box_A)_-$. Since $\kappa \phi'$ strictly decreases degree, $\phi'$ must be a degeneracy. Thus we see that in this case, all faces of $x$ are degenerate. 

In the case $m = n + 1$, then either $x = z \sigma_i$ for some $1 \leq i \leq n$, or $x = z \gamma_{i,\varepsilon}$ for some $1 \leq i \leq n - 1$ and $\varepsilon \in \{0,1\}$. In this case, \cref{some-cubical-identities} shows that two distinct faces of $x$ are equal to $z$; a straightforward calculation shows that all other faces of $x$ are degenerate.
\end{proof}

In many cube categories of interest, there are unique factorizations of maps into standard forms, ie~composites of generators satisfying detailed combinatorial characterizations which facilitate computation; see, for instance, \cite{grandis-mauri} as well as \cite[Fact 7.2]{campion:EZ-cubes}. In the case of cubical sets with diagonals, such factorizations are not available, owing to that category's less convenient combinatorial properties (cf.~\cite[Thm.~8.12]{campion:EZ-cubes}). Nevertheless, we may obtain factorizations in all cube categories which will be of use in studying the corresponding categories of cubical sets; we now define the classes of maps which will appear in these factorizations.

\begin{Def}\label{fix-def}
For $A \subseteq \FS$, $1 \leq i \leq n$, and $\varepsilon \in \{0,1\}$, a map $\phi \colon \Box^m_A \to \Box^n_A$ in $\Box_{A}$ \emph{fixes coordinate $i$ at $\varepsilon$} if $\phi(\avec)_i = \varepsilon$ for all $\avec \in \Box^m_A$. 
A morphism $\phi \colon \Box^m_A \to \Box^n_A$ in $\Box_{A}$ is \emph{active} if it does not fix any coordinate.
\end{Def}

Note that the definition of an active map in $\Box_A$ does not depend on $A$. Among the generating classes of maps listed in \cref{cube-cat}, all except for the faces are active.

\begin{lem}\label{active-poset}
For $A \subseteq \FP$, a map $\phi \colon \Box^m_A \to \Box^n_A$ in $\Box_{A}$ fixes coordinate $i$ at $0$ if and only if $\phi(\ovec)_i = 0$, and fixes coordinate $i$ at $1$ if and only if $\phi(\zvec)_i = 1$.
\end{lem}

\begin{proof}
The ``only if'' direction is immediate from \cref{fix-def}. For the ``if'' direction, suppose $\phi(\zvec)_i = 1$; then because $\phi$ is a poset map and $\zvec$ is minimal, for every $\avec \in \Box_A^m$ we have $\phi(\zvec) \leq \phi(\avec)$, implying in particular that $\phi(\avec)_i = 1$. The proof for the case where $\phi(\ovec)_i = 0$ is similar.
\end{proof}

\begin{cor}\label{active-initial-terminal}
For $A \subseteq \FP$, a map $\phi \colon \Box^m_A \to \Box^n_A$ in $\Box_{A}$ is active if and only if it preserves the initial and terminal elements, ie, $\phi(\zvec_m) = \zvec_n$ and $\phi(\ovec_m) = \ovec_n$. \qed
\end{cor}

\begin{lem}\label{active-face-factor}
For $A \subseteq \FS$, every map $\phi \colon \Box^m_A \to \Box^n_A$ in $\Box_A$ factors as $\bd \psi$, where $\psi$ is active and $\bd$ is a (possibly empty) composite of face maps, with both $\psi$ and $\bd$ unique. Moreover, we have $\bd = \bd_{i_1,\varepsilon_1} \ldots \bd_{i_p,\varepsilon_p} \psi$, where $\psi$ is a unique active map and $i_1 > \ldots > i_p$ is the set of coordinates fixed by $\phi$, with each $i_k$ fixed at $\varepsilon_k$.
\end{lem}

\begin{proof}
We first prove the existence of the stated factorization. We define $\phi' \colon \Box^m_A \to \Box^{n}_A$ as the composite $\bd_{i_1,\varepsilon_1} \cdots \bd_{i_p,\varepsilon_p}\sigma_{i_p} \cdots \sigma_{i_1} \phi$; we will show that this is equal to $\phi$. We may note that the map $\bd_{i_1,\varepsilon_p} \cdots \bd_{i_p,\varepsilon_p}\sigma_{i_p} \cdots \sigma_{i_1} \colon \Box^n_A \to \Box^n_A$ sends a tuple $\avec \in \Box^n_A$ to the tuple obtained by replacing $\avec_{i_k}$ with $\varepsilon_k$ for each $1 \leq k \leq p$. Thus, for any $\avec \in \Box^m_A$ we have $\phi'_{i_k} = \varepsilon_k$ for all $1 \leq k \leq p$; by assumption this means $\phi'(\avec)_{i_k} = \phi(\avec)_{i_k}$. For $j$ not equal to any $i_k$ we likewise have $\phi'(\avec)_{j} = \phi(\avec)_{j}$ as these coordinates are unaffected by composition with $\bd_{i_1,\varepsilon_1}\cdots\bd_{i_p,\varepsilon_p}\sigma_{i_p} \cdots \sigma_{i_1}$. Thus $\phi = \phi'$. 

Now let $\psi = \sigma_{i_p} \cdots \sigma_{i_1} \phi$, so that $\phi = \bd_{i_1,\varepsilon_1} \cdots \bd_{i_p,\varepsilon_p} \psi$. This map does not fix any coordinate, since $\sigma_{i_p} \cdots \sigma_{i_1}$ projects away all the coordinates fixed by $\phi$. Thus $\psi$ is active.

To see that this factorization is unique, suppose that $\bd \psi = \bd' \psi'$, where $\psi$ and $\psi'$ are active and $\bd$ and $\bd'$ are composites of face maps. By \cite[Lem.~4.1]{grandis-mauri}, we may factor $\bd$ as $\bd_{i_1,\varepsilon_1} \ldots \bd_{i_p,\varepsilon_p}$ for some $i_1 > \ldots > i_p$ and $\varepsilon_1, \ldots , \varepsilon_p \in \{0,1\}$, and may similarly factor $\bd'$ as  $\bd_{i'_1,\varepsilon'_1} \ldots \bd_{i'_p,\varepsilon'_p}$. Then the set of coordinates fixed by $\bd \psi$ is precisely $\{i_1,\ldots,i_p\}$, with each $i_k$ fixed at $\varepsilon_k$, and similarly the set of coordinates fixed by $\bd' \psi'$ is precisely $\{i'_1,\ldots,i'_{p'}\}$, with each $i'_k$ fixed at $\varepsilon'_k$. Since $\bd \psi = \bd' \psi'$ by assumption, we thus see that the sets $\{i_1,\ldots,i_p\}$ and $\{i'_1,\ldots,i'_{p'}\}$ are equal; from this it follows that $p = p'$, and that $i_k = i'_k$ and $\varepsilon_k = \varepsilon'_k$ for all $k$. Thus $\bd = \bd'$; since all face maps are monomorphisms, it follows that $\psi = \psi'$.
\end{proof}

\begin{cor}\label{face-intersection}
For $n \geq 2$, $1 \leq i < j \leq n$, and $\varepsilon, \varepsilon' \in \{0,1\}$, a map $\phi \colon \Box^m \to \Box^n$ factors through both $\bd_{i,\varepsilon}$ and $\bd_{j,\varepsilon'}$ if and only if it factors through the composite map $\bd_{j,\varepsilon'}\bd_{i,\varepsilon} = \bd_{i,\varepsilon}\bd_{j-1,\varepsilon'}$. In other words, the following diagram is a pullback:
\[
\begin{tikzcd}
\Box^{n-2}_A \ar[r,"\bd_{j-1,\varepsilon'}"] \ar[d,swap,"\bd_{i,\varepsilon}"] \pullback & \Box^{n-1}_A \ar[d,"\bd_{i,\varepsilon}"] \\
\Box^{n-1}_A \ar[r,"\bd_{j,\varepsilon'}"] & \Box^n_A \\
\end{tikzcd}
\]
\qed
\end{cor}

For a general $A \subseteq \FS$, we write $\partial \Box^n_A$ for the union of all proper faces of $\Box^n_A$ , ie, the subobject of $\Box^n_A$ consisting of all maps into $\Box_A^n$ which factor through a face map on the left. 
We refer to $\bd \Box^n_A$ as the \emph{boundary} of the $n$-cube. Note that for cube categories not containing diagonals, this coincides with the standard definition of the boundary of a representable presheaf on an Eilenberg-Zilber category when $\cSet_A$ is equipped with the Eilenberg-Zilber structure of \cref{cube-EZ}. The subobject of $\Box^n_A$ given by the union of all faces except $\bd_{i,\varepsilon}$ will be denoted $\sqcap^n_{A,i,\varepsilon}$ and referred to as the \emph{$(i, \varepsilon)$-open box} of dimension $n$.

Our factorization results allow us to obtain an explicit description of boundaries and open boxes as coequalizers of coproducts of cubes of lower dimension, analogous to the characterization of boundaries of simplices given in \cite[Sec.~I.2]{goerss-jardine}.

\begin{prop}\label{bdry-colim}
In any cube category $\Box_A$, for $n \geq 2$, the diagram below is a coequalizer.
\[
\begin{tikzcd}
\coprod\limits_{\substack{1 \leq i < j \leq n \\ \varepsilon, \varepsilon' \in \{0,1\}}} \Box^{n-2}_A \arrow[rr,shift left = 0.75ex,"\iota_{j,\varepsilon'} \bd_{i,\varepsilon}"] \arrow[rr, shift right = 0.75ex,swap,"\iota_{i,\varepsilon} \bd_{j-1,\varepsilon'}"] && \coprod\limits_{\substack{1 \leq i \leq n \\ \varepsilon \in \{0,1\}}} \Box^{n-1}_A \arrow[r,"\bd_{i,\varepsilon}"] & \bd \Box^n_A \\
\end{tikzcd}
\]
Moreover, the inclusion $\bd \Box^n_A \hookrightarrow \Box^n_A$ is the  map from the coequalizer induced by the map 
\[
\coprod\limits_{\substack{1 \leq i \leq n \\ \varepsilon \in \{0,1\}}} \Box^{n-1}_A \to \Box^n_A
\]
acting as $\bd_{i,\varepsilon}$ on component $(i,\varepsilon)$.

A similar description holds for any open box $\sqcap^n_{A,i,\varepsilon}$ where the coproducts are taken over all faces (or pairs of faces) other than $\bd_{i,\varepsilon}$.
\end{prop}

\begin{proof}
By a standard result, for any set $X$ and family of subsets $(X_i)_{i \in I}$, there is a coequalizer diagram
\[
\begin{tikzcd}
\coprod\limits_{\substack{i , j \in I \\ i \neq j, X_i \cap X_j \neq \varnothing}} X_i \cap X_j \arrow[r,shift left = 0.75ex] \arrow[r, shift right = 0.75ex] & \coprod\limits_{i \in I} X_i \ar[r] & \bigcup\limits_{i \in I} X_i \\
\end{tikzcd}
\]
where one of the two morphisms between coproducts acts on the $(i,j)$ component as the inclusion of $X_i \cap X_j$ into $X_i$ composed with the coproduct inclusion at $i$, while the other acts as the inclusion into $X_j$ composed with the coproduct inclusion at $j$. Moreover, in such cases the inclusion $\bigcup_{i \in I} X_i \hookrightarrow X$ is the coequalizer map induced by the map $\coprod_{i \in I} X_i \to X$ which acts on the $i$ component as the inclusion of $X_i$.

An analogous result thus holds in any presheaf category for unions of subcomplexes of a given presheaf. In particular, for $n \geq 0$ we may apply this result to $\bd \Box^n_A$ or $\sqcap^n_{A,i,\varepsilon}$, viewed as a union of faces of $\Box^n_A$. In light of \cref{face-intersection} (and the fact that for any $i$ the $(i,0)$ and $(i,1)$ faces of $\Box^n_A$ have empty intersection) we thus obtain the coequalizer diagram given in the statement.
\end{proof}

\subsection{Model structures on cubical sets}

Next we review certain model structures on categories of the form $\cSet_A$ with $A \subseteq \conset$ which model the theories of $\infty$-groupoids and $(\infty,1)$-categories. We begin with a model structure for $\infty$-groupoids, due to Cisinski.

\begin{Def}
A map of cubical sets is a \emph{Kan fibration} if it has the right lifting property with respect to all open box fillings. A cubical set $X$ is a \emph{cubical Kan complex} if the map $X \to \Box^{0}_A$ is a Kan fibration.
\end{Def}

\begin{thm}[Cisinski] \label{Grothendieck-ms-cSet}
  For $A \subseteq \conset$, the category $\cSet_A$ carries a cofibrantly generated model structure, referred to as the Grothendieck model structure, in which the cofibrations are the monomorphisms and the fibrations are the Kan fibrations. In particular, the fibrant objects of this model structure are the cubical Kan complexes.
  
In all of these cases, the Grothendieck model structure on $\cSet_A$ is the test model structure which arises from the fact that $\Box_A$ is a test category.
\end{thm}

\begin{proof}
For the case $A = \varnothing$, see \cite[Thm.~8.4.38]{CisinskiAsterisque}. For other $A \subseteq \conset$, see \cite[Thm.~1.17]{CisinskiUniverses}. (The proof in this reference is for $A = \{\vee\}$, but the proof for other $A$ is identical.)
\end{proof}

Note that the definition of the fibrations in this model structure implies that the open box fillings form a set of generating trivial cofibrations; maps in the saturation of this set will also be referred to as \emph{anodyne maps}.

We next consider models of $(\infty,1)$-categories. In \cite{doherty-kapulkin-lindsey-sattler}, it was shown that each of the categories $\cSet_A$ for $A \subseteq \conset$ carries an analogue of the Joyal model structure on the category $\sSet$ of simplicial sets. To describe these model structures, we must first introduce certain concepts which play a key role in their definition; all of these definitions will be identical regardless of the specific cube category under consideration.

\begin{Def}
We recall standard terminology of the theory of cubical models of $(\infty,1)$-categories.
\begin{itemize}
\item The \emph{critical edge} of $\Box^{n}_A$ with respect to a face $\partial_{i,\varepsilon}$, denoted $\bd^c_{i,\varepsilon}$, is the edge $\Box^1_A \to \Box^n_A$ identified under the isomorphisms $\Box^1_A \cong \Box^0_A \otimes \Box^1_A \otimes \Box^0_A$ and $\Box^n_A \cong \Box^{i-1}_A \otimes \Box^1_A \otimes \Box^{n-i}_A$ with $\overrightarrow{(1-\varepsilon)}_{i-1} \otimes \id \otimes \overrightarrow{(1-\varepsilon)}_{n-i}$. In other words, this is the edge between the two vertices of the $n$-cube having all coordinates other than the $i^{\mathrm{th}}$ equal to $1-\varepsilon$.
\item For $n \geq 2$, $1 \leq i \leq n, \varepsilon \in \{0,1\}$, the $(i,\varepsilon)$-\emph{inner open box}, denoted $\widehat{\sqcap}^{n}_{A,i,\varepsilon}$, is the quotient of $\sqcap^n_{A,i,\varepsilon}$ in which the critical edge $\bd_{i,\varepsilon}^c$ is made degenerate. The $(i,\varepsilon)$-\emph{inner cube}, denoted $\widehat{\Box}^{n}_{A,i,\varepsilon}$, is defined similarly. The $(i,\varepsilon)$-\emph{inner open box inclusion} is the inclusion $\widehat{\sqcap}^{n}_{A,i,\varepsilon} \hookrightarrow \widehat{\Box}^{n}_{A,i,\varepsilon}$.
\item The \emph{invertible interval} $K$ is the cubical set depicted below:
\[
\xymatrix{
  1 \ar[r] \ar@{=}[d] &0 \ar[d] \ar@{=}[r] & 0 \ar@{=}[d] \\
  1 \ar@{=}[r] &1 \ar[r] &0 }
\]
\item The class of \emph{inner anodyne maps} is the saturation of the set of inner open box inclusions.
\item An \emph{inner fibration} is a map having the right lifting property with respect to the inner open box inclusions.
\item An \emph{isofibration} is a map having the right lifting property with respect to the endpoint inclusions $\Box^0_A \hookrightarrow K$.
\item A \emph{cubical quasicategory} is a cubical set $X$ such that the map $X \to \Box^0$ is an inner fibration.
\end{itemize}
\end{Def}

\begin{thm}[{\cite{doherty-kapulkin-lindsey-sattler}}]\label{cubical-Joyal}
For $A \subseteq \conset$, the category $\cSet_A$ carries a model structure, referred to as the cubical Joyal model structure, in which:
\begin{itemize}
\item The cofibrations are the monomorphisms;
\item The fibrant objects are the cubical quasicategories;
\item A map between fibrant objects is a fibration if and only if it is both an inner fibration and an isofibration. \qed
\end{itemize}
\end{thm}

From this, we can see that the inner open box inclusions and the endpoint inclusions $\Box^0_A \hookrightarrow K$ form a set of \emph{pseudo-generating trivial cofibrations}, ie, that fibrations with fibrant codomain are characterized by the right lifting property with respect to these maps.

In both the Grothendieck and cubical Joyal model structures, cylinder objects may be constructed using the geometric product. In the Grothendieck model structure, a cylinder on $X$ is given by $X \otimes \Box^1_A$, while in the cubical Joyal model structure, a cylinder on $X$ is given by $X \otimes K$. This gives rise to a natural concept of homotopy in each model structure; by standard constructions, weak equivalences can then be defined as maps inducing isomorphisms on homotopy classes of maps into fibrant objects. 

More generally, we have the following result demonstrating the compatibility of these model structures on cubical sets with the geometric product.

\begin{thm} \label{geometric-monoidal-models}
    For $A \subseteq \conset$, the Grothendieck and cubical Joyal model structures on $\cSet_A$ are monoidal with respect to the geometric product.
\end{thm}

\begin{proof}
    For the Grothendieck model structure on $\cSet_\vee$, this is part of \cite[Thm.~1.7]{CisinskiUniverses}; identical proofs apply for other $A$. For the cubical Joyal model structure, this is \cite[Cor.~4.11]{doherty-kapulkin-lindsey-sattler}.
\end{proof}

The following result will allow us to prove many statements about the Grothendieck model structure as immediate consequences of analogous statements about the cubical Joyal model structure.

\begin{prop}[{\cite[Prop.~4.23(i)]{doherty-kapulkin-lindsey-sattler}}]\label{cubical-Joyal-localization}
For $A \subseteq \conset$, the Grothendieck model structure on $\cSet_A$ is a localization of the cubical Joyal model structure. In particular, every weak equivalence in the cubical Joyal model structure is a weak equivalence in the Grothendieck model structure. \qed
\end{prop}

\section{Comparisons between categories of cubical sets} \label{sec weak sym}

Our task in this section is to develop the technical tools which we will use to relate the categories of cubical sets on which the Grothendieck and cubical Joyal model structures have been established to those having symmetries or diagonals. In \cref{sec i-shriek}, we introduce the adjoint triples which will be used in these comparisons, and prove some of their basic properties. In \cref{sec N-construction} we describe our main combinatorial construction, the \emph{standard decomposition cubes}, and analyze some of their basic properties. In \cref{sec N-closed}, we show how the standard decomposition cubes can be used to prove that certain morphisms of cubical sets are trivial cofibrations, and we discuss some basic consequences of this result, including that the geometric product of cubical sets is symmetric up to a zigzag of natural weak equivalences in the cubical Joyal model structure.

\subsection{The left Kan extension functor} \label{sec i-shriek}

For $A \subseteq B \subseteq \FS$ we have an inclusion of cube categories $i \colon \Box_A \hookrightarrow \Box_B$. This induces a pre-composition functor on the corresponding presheaf categories, denoted $i^* \colon \cSet_B \to \cSet_A$. One may view this as a forgetful functor: for $X \in \cSet_B$, the cubical set $i^* X \in \cSet_A$ has the same set of cubes in each dimension as $X$, but lacks the structure maps corresponding to morphisms of $\Box_B$ which are not present in $\Box_A$. For instance, if $A = \varnothing$ and $B = \{\wedge\}$, then the connections of a cubical set $X \in \cSet_{\wedge}$ still appear as cubes of $i^* X$, but can no longer be identified as connections; they are, in general, non-degenerate.

The operations of left and right Kan extension define left and right adjoint functors to $i^*$, respectively; thus we obtain an adjoint triple $i_! \adjoint i^* \adjoint i_*$. In \cite{doherty:without-connections}, these adjoint triples are considered in the setting of model structures for $(\infty,n)$-categories on categories $\cSet^+_A$ of cubical sets with markings. The techniques used in that paper, and results thereby obtained, generalize easily to the unmarked setting. 

Specifically, the key idea of the proof of \cite[Prop.~2.2]{doherty:without-connections} is to show that for any marked cubical set $X$, the unit $\eta_X \colon X \to i^* i_! X$ is anodyne, by verifying that it has the left lifting property with respect to \emph{comical fibrations}, the naive fibrations of the model structures studied in that paper, and this is done by explicitly constructing lifts via filling of open boxes having specified faces marked. In the setting of unmarked cubical sets modelling $(\infty,1)$-categories or $\infty$-groupoids, one can likewise verify that $\eta_X$ has the right lifting property with respect to inner fibrations, by explicitly constructing lifts via filling inner open boxes. At any step of the proof of \cite[Prop.~2.9]{doherty:without-connections} and the results on which it depends in which it is verified that an open box is \emph{comical}, ie~that its faces satisfy the marking conditions given in \cite[Def.~2.1]{doherty-kapulkin-maehara}, the given proof specializes to a proof that its critical edge is degenerate (note that the critical edge is the unique edge which is required to be marked by these conditions).
One may then use the analogue of \cite[Prop.~2.29]{doherty:without-connections} to prove analogues of \cite[Thm.~2.25 \& Cor.~2.29]{doherty:without-connections} via arguments of homotopical algebra similar to those used in \cite{doherty:without-connections}.

\begin{prop}[cf~{\cite[Prop.~2.9]{doherty:without-connections}}] \label{unit-tcof-empty}
    For $A \subseteq B \subseteq \conset$, the unit of $i_! : \cSet_A \rightleftarrows \cSet_B : i^*$ is a trivial cofibration in the cubical Joyal model structure on $\cSet_A$. \qed
\end{prop}

\begin{thm}[cf~{\cite[Thm.~2.25 \& Cor.~2.29]{doherty:without-connections}}] \label{Quillen-equiv-empty}
    For $A \subseteq B \subseteq \conset$, both of the adjunctions $i_! : \cSet_A \rightleftarrows \cSet_B : i^*$ and $i^* : \cSet_B \rightleftarrows \cSet_A : i_*$ are Quillen equivalences, where $\cSet_A$ and $\cSet_B$ are equipped with either their respective Grothendieck or cubical Joyal model structures. Moreover, in both cases, both of the left adjoints $i_!, i^*$ create weak equivalences. \qed
\end{thm}

Note that although the Quillen equivalence mentioned in \cref{Quillen-equiv-empty} for the cubical Joyal model structures is stated as \cite[Thm.~A.24]{doherty:without-connections}, the proof given in that paper is incorrect, as it relies on the  assumption that the forgetful functor $|-| \colon \cSet^+_A \to \cSet_A$ (which forgets markings on cubes) is left Quillen; this is not the case for the cubical Joyal model structure and the model structure on marked cubical sets considered in that work, as is correctly observed earlier in the same appendix of that paper.

The adjunctions $i_! \adjoint i^*$ will be of the greatest relevance to our work; thus our immediate task is to characterize the left adjoints $i_!$. A standard result concerning left Kan extension functors between presheaf categories gives us the following.

\begin{prop}\label{i-shriek-on-reps}
For any $A \subseteq B \subseteq \FS$, the following diagram commutes up to natural isomorphism, where the vertical functors are Yoneda embeddings:
\[
\begin{tikzcd}
\Box_A \ar{r}{i} \ar{d} & \Box_B \ar{d} \\
\cSet_A \ar{r}{i_!} & \cSet_B \\
\end{tikzcd}
\]
In particular, for all $n \geq 0$ we have $i_! \Box_A^n \cong \Box_B^n$.
\qed
\end{prop}

Roughly speaking, for $X \in \cSet_A$, $i_! X$ is obtained by freely adding the structure maps of $\Box_B$ to $X$ in a manner compatible with the structure maps of $X$. More precisely, by the density theorem and the fact that $i_!$ preserves colimits as a left adjoint, we have the following description of $i_!X$:
\[
i_! X = \colim\limits_{\Box^n_A \to X} \Box^n_B 
\]
From this, together with the standard description of a colimit as a coequalizer of coproducts, we obtain the following.

\begin{prop}\label{i-shriek-characterization}
Let $A \subseteq B \subseteq \FS$. For $X \in \cSet_A$, an $n$-cube of $i_! X$ is determined by a cube $x \in X_m$ for some $m$ and a map $\psi \colon \Box^n_B \to \Box^m_B$; we write this as $x \psi$, or simply $x$ in the case $\psi = \id$. These are subject to the identification $(x \phi) \psi = x (\phi \psi)$ for $x$ as above, $\phi \colon \Box^n_A \to \Box^m_A$ and $\psi \colon \Box^{n'}_B \to \Box^n_B$. Structure maps of $\Box_B$ act by pre-composition.

Given a map $f \colon X \to Y$ in $\cSet_A$, the map $i_! f \colon i_! X \to i_! Y$ sends each cube $x \phi$ to $f(x) \phi$. \qed
\end{prop}

Analyzing the adjunction $i_! \adjoint i^*$ in view of this characterization of $i_!$, we obtain the following description of its unit.

\begin{cor}\label{unit-characterization}
For $X \in \cSet_A$ as above, the unit map $\eta \colon X \to i^* i_! X$ sends each cube $x$ of $X$ to $x$, viewed as a cube of $i^* i_! X$. \qed
\end{cor}

We likewise obtain convenitent characterizations of the actions of $i_!$ on some of the key objects and morphisms used in the study of the cubical Joyal model structures.

\begin{lem}\label{i-shriek-open-box}
For any $A \subseteq B \subseteq \FS$, the image under $i_!$ of a boundary inclusion, (inner) open box filling, or endpoint inclusion into $K$ in $\cSet_A$ is the corresponding map in $\cSet_B$.
\end{lem}

\begin{proof}
In all cases, this follows from the fact that the relevant map is induced by a particular colimit defined according to a formula which does not depend on $A$.

For inclusions of boundaries and open boxes, we note that by \cref{i-shriek-on-reps} and the fact that $i_!$ preserves colimits as a left adjoint, $i_!$ sends the colimit diagrams and maps of \cref{bdry-colim} to their analogues in $\cSet_B$. 

For inner open box fillings, we note that the inner cube $\hBox^n_{A,i,\varepsilon}$ is defined by the following pushout diagram, where the top horizontal map picks out the critical edge with respect to $\bd_{i,\varepsilon}$.
\[
\begin{tikzcd}
\Box^1_A \ar[r,"\bd^c_{i,\varepsilon}"] \ar[d] \pushout & \Box^n_A \ar[d] \\
\Box^0_A \ar[r] & \hBox^n_{A,i,\varepsilon} \\
\end{tikzcd}
\]
The inner open box $\hcap^n_{A,i,\varepsilon}$ is defined by a similar pushout diagram, and the inclusion $\hcap^n_{A,i,\varepsilon} \hookrightarrow \hBox^n_{A,i,\varepsilon}$ is the map between pushouts induced by the inclusion $\sqcap^n_{A,i,\varepsilon} \hookrightarrow \Box^n_A$. The fact that $i_!$ preserves this inner open box inclusion thus follows from the previously established fact that it preserves open box inclusions, together with a further application of \cref{i-shriek-on-reps} and the fact that it preserves colimits.

For endpoint inclusions into $K$, we may make a similar argument based on an explicit construction of $K$ as a colimit of representable cubical sets.
\end{proof}

\begin{cor}\label{i-shriek-monos}
For $A \subseteq B \subseteq \FS$ with $A \subseteq \conset$, the functor $i_! \colon \cSet_A \to \cSet_B$ preserves monomorphisms.
\end{cor}

\begin{proof}
In these cases, the class of monomorphisms in $\cSet_A$ is the saturation of the set of boundary inclusions (see \cite[Lem.~1.32]{doherty-kapulkin-lindsey-sattler}), so this is immediate from \cref{i-shriek-open-box}.
\end{proof}

In the case where $A \subseteq \conset$, the convenient combinatorial properties of $\cSet_A$ allow us to characterize the cubes of $i_! X$ in terms of the non-degenerate cubes of $X$.

\begin{prop}\label{i-shriek-pairs}
For $A \subseteq B \subseteq \FS$ with $A \subseteq \conset$, and $X \in \cSet_A$, each cube of $i_! X$ is equal to $x\phi$ for a unique non-degenerate cube $x$ of $X$ and a unique active map $\phi$ in $\Box_B$.
\end{prop}

\begin{proof}
Since $\Box_A$ is an Eilenberg-Zilber category by \cref{cube-EZ}, we may proceed by induction on skeleta. We first note that $i_! \Box^n_A = \Box^n_B$ by \cref{i-shriek-on-reps}. The statement for the case $X = \Box^n_A$ thus follows from \cref{active-face-factor} and the fact that the non-degenerate cubes of $\Box^n_A$ are precisely the composites of face maps (including the identity, viewed as an empty composite). Now suppose that the statement holds for some $X \in \cSet_A$, and that $Y$ is obtained from $X$ by adjoining a single non-degenerate $n$-cube $y$; in other words, suppose we have a pushout diagram
\[
\begin{tikzcd}
\bd \Box^n_A \ar[r]  \ar[tail,d] \pushout & X \ar[tail,d] \\
\Box^n_A \ar[r,"y"] & Y \\
\end{tikzcd}
\]
in $\cSet_A$. Then $i_!$ preserves this pushout as a left adjoint. Applying  \cref{i-shriek-open-box} and restricting to level $m$ for a given $m \geq 0$, we have a pushout diagram
\[
\begin{tikzcd}
(\bd \Box^n_B \ar[r])_m  \ar[tail,d] \pushout & (i_! X)_m \ar[tail,d] \\
(\Box^n_B)_m \ar[r,"i_! y"] & (i_! Y)_m \\
\end{tikzcd}
\]
in $\Set$. 

Thus $(i_! Y)_m \cong (i_! X)_m \sqcup ((\Box^n_B)_m \setminus (\bd \Box^n_B)_m)$. By the induction hypothesis, every cube in the left component of this disjoint union is equal to $x \phi$ for a unique $k \geq 0$, a unique non-degenerate $x \colon \Box^k_A \to X$, and a unique active $\phi \colon \Box^m_B \to \Box^k_B$. To characterize cubes in the second coproduct component, we may note that the complement $((\Box^n_B)_m \setminus (\bd \Box^n_B)_m)$ consists precisely of the active maps $\phi \colon \Box^m_B \to \Box^n_B$, and each such map, viewed as an $m$-cube of $\Box^n_B$, is sent by $i_! y$ to $y \phi$. Thus the cubes $y \phi$ are distinct as elements of $(i_! Y)_m$, both from each other and from the cubes of $(i_! X)_m$. Thus the statement holds for $Y$.

Finally, we must show that the stated property is preserved by transfinite composition. Let $\kappa$ be a limit ordinal, and consider a transfinite composite of inclusions $X_{\lambda} \hookrightarrow X_{\lambda'}$ for $\lambda < \lambda' < \kappa$, such that the statement holds for each $X_{\lambda}$. Denote the colimit object by $X$; then $X$ is the union of all the $X_\lambda$, with the inclusion $X_\lambda \hookrightarrow X$ given by the colimit map. Then $i_!$ preserves this transfinite composite as a left adjoint, so that $i_! X$ is similarly the union of the $i_! X_\lambda$. Thus every cube $y$ of $i_! X$ is contained in $i_! X_\lambda$ for some $\lambda$; it follows that $y = x \phi$ for some non-degenerate cube $x$ of $X_\lambda \subseteq X$ and some active map $\phi$. To see that $x$ and $\phi$ are unique, suppose that $x \phi = x' \phi'$, where both $x$ and $x'$ are non-degenerate and both $\phi$ and $\phi'$ are active. Then there is some $\lambda$ such that both $x$ and $x'$ are contained in $X_\lambda$. Since $i_! X_\lambda \hookrightarrow i_! X$ is a monomorphism by \cref{i-shriek-monos}, we have $x \phi = x' \phi'$ in $X_\lambda$ as well. Thus $x = x'$ and $\phi = \phi'$ by the uniqueness property in $X_\lambda$.
\end{proof}

\begin{cor}\label{i-shriek-nondegen}
For $A, B, X$ as above, the non-degenerate cubes of $i^* i_! X$ are precisely those of the form $x \phi$ where $x$ is a non-degenerate cube of $X$ and $\phi$ is an active map not factoring on the right through any degeneracy in $\Box_A$. \qed
\end{cor}

\begin{cor}\label{unit-mono}
For $A, B, X$ as above, the unit map $\eta \colon X \to i^* i_! X$ is a monomorphism. In particular, if $x, x' \colon \Box^n_A \to X$ are distinct, then the corresponding cubes of $i^* i_! X$ are distinct as well.
\end{cor}

\begin{proof}
This is immediate from \cref{unit-characterization,i-shriek-pairs}.
\end{proof}

Finally, we consider the compatibility of $i_!$ with the geometric product of cubical sets.

\begin{prop}\label{i-shriek-monoidal}
For any $A \subseteq B \subseteq \FS$, the adjunction $i_! : \cSet_A \rightleftarrows \cSet_B : i^*$ is strong monoidal with respect to the geometric product.
\end{prop}

\begin{proof}
Following the proof structure of \cite[Prop.~4.4]{isaacson:symmetric}, we observe that by \cref{i-shriek-on-reps}, $i_!$ is the left Kan extension of the composite of $i$ with the Yoneda embedding on $\Box_B$, which sends each object $\Box_A^n$ to $\Box^n_B$. By \cite[Thm.~5.1]{im-kelly:convolution}, the statement thus follows from the fact that this composite is strong monoidal.
\end{proof}

\subsection{Standard decomposition cubes}\label{sec N-construction}

Our next task is to define and study a combinatorial construction which is of use in relating different kinds of cubical sets, and which will play a key role in proving the main results of this section. For concreteness, throughout the remainder of this subsection we fix $A \subseteq B \subseteq \FS$ with $\wedge \in A$. Of course, for the case $\vee \in A$, we could make parallel arguments involving negative connections.

\begin{Def}\label{N-def}
For a map $\phi \colon \Box^m_B \to \Box^n_B$, we define the \emph{standard decomposition cube}  $N_k(\phi)$ to be the composite 
\[
\gamma_{n,1} (\phi \otimes \Box^1_B \otimes \Box^k_B) \colon \Box^{m+1+k}_B \to \Box^{n+k}_B
\]
\end{Def}

Note that we have written $\Box^1_B \otimes \Box^k_B$ above rather than $\Box^{1+k}_B$ because coordinate $m+1$ will play a distinct role in our computations.

We will also have use for the following more explicit characterization of maps $N_k(\phi)$.

\begin{lem}\label{N-explicit}
Given an $(m+1+k)$-tuple $(\avec,b,\overrightarrow{c}) = (a_1,\ldots,a_m,b,c_1,\ldots,c_k) \in \Box^{m+1+k}$, $N_k(\phi)$ sends $(\avec,b,\overrightarrow{c})$ to the $(n+k)$-tuple
\[
(\phi(\avec)_1,\ldots,\phi(\avec)_{n-1}, \phi(\avec)_n \wedge b, c_1, \ldots, c_k)
\]
in $\Box^{n+k}$. \qed
\end{lem}

\begin{cor}\label{N-k-tensor}
For all $j, k \geq 0$ and $\phi \colon \Box^m_B \to \Box^n_B$, we have $N_{j+k}(\phi) = N_j(\phi) \otimes \Box^k_B$. In particular, $N_k(\phi) = N_0(\phi) \otimes \Box^k_B$ for all $k$. \qed
\end{cor}

\begin{lem}\label{N-active}
    If $\phi$ is active, then so is $N_k(\phi)$ for any $k$.
\end{lem}

\begin{proof}
From \cref{N-explicit}, it is immediate that a map of the form $N_k(\phi)$ does not fix coordinates $1$ through $n-1$ if $\phi$ does not, and does not fix coordinates $n+1$ through $n + 1 + k$ regardless of $\phi$. It thus remains to consider coordinate $n$. We may note that $N_k(\phi)(\avec,1,\overrightarrow{c})_n = \phi(\avec)_n \wedge 1 = \phi(\avec)_n$. Therefore, if $\phi$ does not fix coordinate $n$, then neither does $N_k(\phi)$.
\end{proof}

The purpose of this definition is to facilitate proofs involving open box filling; in the most important cases, we will fill open boxes in subcomplexes of $i^* \Box^n_B$ whose interiors will be standard decomposition cubes, with the missing face corresponding to $\phi \otimes \Box^k_B$. Moreover, we will see that this open box is inner in the case where $B \subseteq \FP$.

The rough intuition behind the definition of the standard decomposition cubes is as follows. We may view $\phi \otimes \Box^k_B$ as a network of paths in the $(n+k)$-cube from its initial vertex $(\phi \otimes \Box^k_B)(\zvec)$ to its terminal vertex $(\phi \otimes \Box^k_B)(\ovec)$. The $(m+1,0)$-face of $N_k(\phi)$ is $\bd_{m,0} \sigma_m (\phi \otimes \Box^k_B)$, which corresponds to the network of paths which proceeds as prescribed by $\phi$ on the first $(m-1)$ coordinates, while leaving the $m^{\mathrm{th}}$ coordinate fixed at 0. The edges connecting the $(m+1,0)$-face to the $(m+1,1)$-face then proceed from each vertex of $\bd_{m,0} \sigma_m (\phi \otimes \Box^k_B)$ to the corresponding vertex of $\phi \otimes \Box^k_B$.
Thus the $(m+1,1)$-open box on $N_k(\phi)$ is obtained by ``separating out'' the $m^{\mathrm{th}}$ component of $\phi \otimes \Box^k_B$ into an extra dimension.

\begin{examples}\label{N-examples}
In order to further develop the intuition behind the definition above, we illustrate various cubes of the form $N_0(\phi)$. 

\begin{itemize}
\item For $\delta_1 \colon \Box^1_B \to \Box^2_B$, mapping $a$ to $(a,a)$,the cube $N_0(\delta_1) \colon \Box^2_B \to \Box^2_B$, corresponding to the map $(a,b) \mapsto (a,a \wedge b)$, is pictured below:
\[
\begin{tikzcd}
(0,0) \ar[r] \ar[d,equal] & (1,0) \ar[d] \\
(0,0) \ar[r] & (1,1) \\
\end{tikzcd}
\]
We may view this face as witnessing a  factorization of the diagonal edge $(0,0) \to (1,1)$ as a composite of the edges $(0,0) \to (1,0)$ and $(1,0) \to (1,1)$. Viewing the diagonal edge as a path of length 1 from $(0,0)$ to $(1,1)$ in the poset $\{0 \leq 1\}^2$, we may think of this factorization as being obtained by ``separating out'' the two components of this path: first advancing from $0$ to $1$ in the first component while keeping the second fixed, then advancing from $0$ to $1$ in the second component while keeping the first fixed. Note that $\delta_1$ appears as the $(2,1)$-face of this cube, and that the critical edge with respect to this face is degenerate.
\item For $\lambda_{21} \colon \Box^2_B \to \Box^2_B$, mapping $(a_1,a_2)$ to $(a_2,a_1)$, the cube $N_0(\lambda_{21})$, corresponding to the map $(a_1, a_2, b) \mapsto (a_2, a_1 \wedge b)$, is depicted below:
\[
\begin{tikzcd}
(0,0) \ar[rr,equal] \ar[dd] \ar[dr,equal] &  & (0,0) \ar[dd] \ar[dr] \\
& (0,0) \ar[rr,crossing over] & & (0,1) \ar[dd] \\
(1,0) \ar[rr,equal] \ar[dr,equal] & & (1,0) \ar[dr] \\
& (1,0) \ar[rr] \ar[from=uu,crossing over] & & (1,1) \\
\end{tikzcd}
\]
Similarly to the previous example, we may view the cube $\lambda_{21}$ as a family of paths from $(0,0)$ to $(1,1)$. Then the $(3,0)$ face of this cube is obtained by proceeding from $0$ to $1$ in the first component as prescribed by this family of paths, while keeping the second component fixed at 0. The edges connecting the $(3,0)$ face to the $(3,1)$ face then advance the second component from $0$ to $1$ whenever a 1 appears in the second component of the corresponding vertex of $\phi$. Note that $\lambda_{21}$ appears as the $(3,1)$-face of this cube, and that the critical edge with respect to this face is degenerate.
\item For $\rho_1 \colon \Box^1_B \to \Box^1_B$, the cube $N_0(\rho_1)$, corresponding to the map $(a,b) \mapsto (1-a) \wedge b$, is pictured below:
\[
\begin{tikzcd}
0 \ar[r,equal] \ar[d] & 0 \ar[d,equal] \\
1 \ar[r] & 0 \\
\end{tikzcd}
\]
Here the $(2,0)$ face holds the unique component fixed at 0, while the edges adjacent to this face proceed from 0 to 1 whenever the corresponding vertex of $\rho_1$ is 1. In this case, because the initial vertex of $\rho_1$ is 1, this does not result in the critical edge with respect to the $(2,1)$-face being degenerate, unlike the critical edges with respect to the $(m+1,1)$-faces in the previous examples.
\end{itemize}
\end{examples}

We next study various combinatorial properties of standard decomposition cubes which will be of use in our open box filling constructions. We begin with certain identities satisfied by standard decomposition cubes associated to composite maps.

\begin{lem}\label{N-identities}
For any $\phi \colon \Box^m_B \to \Box^n_B$ and any $k \geq 0$, the following identities are satisfied:
\begin{enumerate}
\item \label{N-degen} $N_k(\phi \sigma_i) = N_k(\phi)\sigma_i$ for all $1 \leq i \leq m + 1$;
\item \label{N-con} $N_k(\phi \gamma_{i,\varepsilon}) = N_k(\phi) \gamma_{i,\varepsilon}$ for all $1 \leq i \leq m$ and $\varepsilon \in \{0,1\}$;
\item \label{N-low-face} $N_k(\bd_{i,\varepsilon}\phi) = \bd_{i,\varepsilon}N_k(\phi)$ for $1 \leq i \leq n$ and $\varepsilon \in \{0,1\}$;
\item \label{N-high-zero-face-new} $N_k(\bd_{n+1,0}\phi) = \bd_{n+1,0}\sigma_n (\phi \otimes \Box_B^1 \otimes \Box_B^k)$;
\item \label{N-high-one-face-new} $N_k(\bd_{n+1,1}\phi) = \phi \otimes \Box_B^1 \otimes \Box_B^k$.
\end{enumerate}
\end{lem}

\begin{proof}
For \cref{N-degen}, we compute:
\begin{align*}
N_k(\phi \sigma_i) & = \gamma_{n,1}(\phi \sigma_i \otimes \Box_B^1 \otimes \Box_B^k) \\
& = \gamma_{n,1} (\phi \otimes \Box_B^1 \otimes \Box_B^k) (\sigma_i \otimes \Box_B^1 \otimes \Box_B^k) \\
& = N_k(\phi) \sigma_i
\end{align*}
The calculation for \cref{N-con} is similar.
For \cref{N-low-face,N-high-zero-face-new,N-high-one-face-new} we compute:
\begin{align*}
N_k(\bd_{i,\varepsilon} \phi) & = \gamma_{n+1,1}(\bd_{i,\varepsilon} \phi \otimes \Box_B^1 \otimes \Box_B^k) \\
& = \gamma_{n+1,1} (\bd_{i,\varepsilon} \otimes \Box_B^1 \otimes \Box_B^k)  (\phi \otimes \Box_B^1 \otimes \Box_B^k) \\
& = \gamma_{n+1,1} \bd_{i,\varepsilon} (\phi \otimes \Box_B^1 \otimes \Box_B^k) \\
\end{align*}
For $1 \leq i \leq n$, this is:
\begin{align*}
 \bd_{i,\varepsilon} \gamma_{n,1} (\phi \otimes \Box_B^1 \otimes \Box_B^k) & = \bd_{i,\varepsilon} N_k(\phi) \\
\end{align*}
This proves \cref{N-low-face}. In the case $i = n+1$, we may apply \cref{some-cubical-identities} to the composite $\gamma_{n+1,1}\bd_{n+1,\varepsilon}$, thus obtaining \cref{N-high-zero-face-new,N-high-one-face-new} follow.
\end{proof}

We next analyze the faces of standard decomposition cubes in detail.

\begin{lem}\label{N-faces}
For any map $\phi \colon \Box^m_B \to \Box^n_B$, we may characterize the faces of $N_k(\phi)$ as follows.

\begin{itemize}
\item For $i \leq m$ and $\varepsilon \in \{0,1\}$, we have $N_k(\phi)\bd_{i,\varepsilon} = N_k(\phi \bd_{i,\varepsilon})$. Moreover, in this case one of the following holds:
\begin{enumerate}
\item \label{face-case-bdry} %
$N_k(\phi)\bd_{i,\varepsilon}$ factors through a face map on the left;
\item \label{face-case-tensor} $N_k(\phi)\bd_{i,\varepsilon} = \psi \otimes \Box_B^1 \otimes \Box_B^k$ for some active map $\psi \colon \Box^{m-1}_B \to \Box^{n-1}_B$;
\item \label{face-case-N} $N_k(\phi)\bd_{i,\varepsilon} = N_k(\psi)$ for some active map $\psi \colon \Box^{m-1}_B \to \Box^n_B$.
\end{enumerate}
\item $N_k(\phi) \bd_{m+1,0} = \bd_{n,0} \sigma_n \phi \otimes \Box^k_B$.
\item $N_k(\phi) \bd_{m+1,1} = \phi \otimes \Box^k_B$.
\item For $k \geq 1$ and $i = m + 1 + p, p \geq 1$, and $\varepsilon \in \{0,1\}$, we have $N_k(\phi)\bd_{i,\varepsilon} = \bd_{n+p,\varepsilon} N_{k-1}(\phi)$.
\end{itemize}
\end{lem}

\begin{proof}
We first consider the case $i \leq m$. In this case, using the functoriality of the monoidal product, we may compute:
\begin{align*}
N(\phi) \bd_{i,\varepsilon} & = \gamma_{n,1}(\phi \otimes \Box^1_B \otimes \Box^k_B) \bd_{i,\varepsilon} \\
& = \gamma_{n,1}(\phi \otimes \Box^1_B) (\bd_{i,\varepsilon} \otimes \Box^1_B \otimes \Box^k_B) \\
& = \gamma_{n,1}(\phi \bd_{i,\varepsilon} \otimes \Box^1_B \otimes \Box^k_B) \\
& = N_k(\phi \bd_{i,\varepsilon}) \\
\end{align*}

To analyze this case further, we consider the set of coordinates fixed by $\phi \bd_{i,\varepsilon} \colon \Box^{m-1}_B \to \Box^{n}_B$. If this map fixes some coordinate $j < n$ at $\varepsilon'$, then by \cref{active-face-factor} we have $\phi \bd_{i,\varepsilon} = \bd_{j,\varepsilon'} \psi$ for some $\psi \colon \Box^{m-1}_B \to \Box^{n-1}_B$. Thus we may compute:
\begin{align*}
N_k(\phi \bd_{i,\varepsilon}) & = N_k(\bd_{j,\varepsilon'} \psi) \\
& = \gamma_{n,1} (\bd_{j,\varepsilon'} \psi \otimes \Box^1_B \otimes \Box^k_B) \\
& = \gamma_{n,1} \bd_{j,\varepsilon'} (\psi \otimes \Box^1_B \otimes \Box^k_B) \\
& = \bd_{j,\varepsilon'} \gamma_{n-1,1} (\psi \otimes \Box^1_B \otimes \Box^k_B) \\
\end{align*}
We can make a similar calculation in the case where $\phi \bd_{i,\varepsilon}$ fixes coordinate $n$ at $0$. So in these cases, condition \ref{face-case-bdry} holds.

Next suppose that the only coordinate fixed by $\phi \bd_{i,\varepsilon}$ is $n$, and this is fixed at $1$. In this case, again applying \cref{active-face-factor}, we have $\phi \bd_{i,\varepsilon} = \bd_{n,1} \psi$ for an active map $\psi \colon \Box^{m-1}_B \to \Box^{n-1}_B$. Thus we may compute:
\begin{align*}
N_k(\phi \bd_{i,\varepsilon}) & = N_k(\bd_{n,1} \psi) \\
& = \gamma_{n,1} (\bd_{n,1} \psi \otimes \Box^1_B \otimes \Box^k_B) \\
& = \gamma_{n,1} \bd_{n,1} (\psi \otimes \Box^1_B \otimes \Box^k_B) \\
& = \psi \otimes \Box^1_B \otimes \Box^k_B \\
\end{align*}
Thus condition \ref{face-case-tensor} holds in this case.

Finally, we consider the case where $\phi \bd_{i,\varepsilon}$ does not fix any coordinate. By definition, this means that $\phi \bd_{i,\varepsilon}$ is active, so condition \ref{face-case-N} holds in this case.

Now we consider the case $i = m + 1$. For this case, we first note that $\bd_{m+1,\varepsilon} = \Box^m_B \otimes \bd_{1,\varepsilon} \otimes \Box^k_B$. Now observe that by  the functoriality of the monoidal product, we have the following commuting diagram:
\[
\begin{tikzcd}
\Box^{m}_B \otimes \Box^0_B \otimes \Box^k_B \arrow{rr}{\Box_B^m \otimes \bd_{1,\varepsilon} \otimes \Box_B^k} \arrow[d,swap,"\phi \otimes \Box_B^0 \otimes \Box_B^k"] & & \Box_B^{m} \otimes \Box_B^1 \otimes \Box_B^k \arrow{d}{\phi \otimes \Box_B^1 \otimes \Box_B^k} \\
\Box_B^n \otimes \Box_B^0 \otimes \Box_B^k \arrow{rr}{\Box_B^n \otimes \bd_{1,\varepsilon} \otimes \Box_B^k} & & \Box_B^{n} \otimes \Box_B^1 \otimes \Box_B^k \\
\end{tikzcd}
\]
Thus we see that $(\phi \otimes \Box_B^1 \otimes \Box_B^k) \bd_{m+1,\varepsilon} = \bd_{n+1,\varepsilon} (\phi \otimes \Box_B^k)$. Post-composing with $\gamma_{n,1}$ on the left side of this equation yields $N_k(\phi) \bd_{m+1,\varepsilon}$, while by \cref{some-cubical-identities}, post-composing with $\gamma_{n,1}$ on the right side yields either $\phi \otimes \Box_B^k$ or $\bd_{n,0} \sigma_n (\phi \otimes \Box_B^k)$ depending on the value of $\varepsilon$.

Finally, we consider the case $i > m + 1$, ie~$i = m + 1 + p$ for some $p > 0$; note that this case can only occur for $k \geq 1$. In this case, we have $\bd_{m+1+p,\varepsilon} = \Box_B^m \otimes \Box_B^1 \otimes \bd_{p,\varepsilon}$. Thus we obtain a commuting diagram, similarly to the previous case: 
\[
\begin{tikzcd}
\Box_B^{m} \otimes \Box_B^1 \otimes \Box_B^{k-1} \arrow{rr}{\Box_B^m \otimes \Box_B^1 \otimes \bd_{p,\varepsilon}} \arrow[d,swap,"\phi \otimes \Box_B^1 \otimes \Box_B^{k-1}"] & & \Box_B^{m} \otimes \Box_B^1 \otimes \Box_B^k \arrow{d}{\phi \otimes \Box_B^1 \otimes \Box_B^k} \\
\Box_B^n \otimes \Box_B^1 \otimes \Box_B^{k-1} \arrow{rr}{\Box_B^n \otimes \Box_B^1 \otimes \bd_{p,\varepsilon}} & & \Box_B^{n} \otimes \Box_B^1 \otimes \Box_B^k \\
\end{tikzcd}
\]
Thus we see that $(\phi \otimes \Box_B^1 \otimes \Box_B^k)\bd_{m+1+p,\varepsilon} = \bd_{n+1+p,\varepsilon}(\phi \otimes \Box_B^1 \otimes \Box_B^{k-1})$. Post-composing with $\gamma_{n,1}$, we obtain $\bd_{n+p,\varepsilon} \gamma_{n,1} (\phi \otimes \Box_B^1 \otimes \Box_B^{k-1}) = \bd_{n+p,\varepsilon}N_{k-1}(\phi)$.
\end{proof}

The following lemma will be used to show that when $\phi$ is a poset map, the open box which will be filled to obtain $\phi \otimes \Box^k_B$ and $N_k(\phi)$ is inner.

\begin{lem}\label{N-crit-edge}
If $B \subseteq \FP$, then for any active map $\phi \colon \Box^m_{B} \to \Box^n_{B}$ and any $k \geq 0$, the critical edge of $N_k(\phi)$ with respect to $\bd_{m+1,1}$ is degenerate.
\end{lem}

\begin{proof}
The critical edge with respect to $\bd_{m+1,1}$, viewed as a map $\Box_B^1 \to \Box_B^m \otimes \Box_B^1 \otimes \Box_B^k$ in $\Box_A$, sends $a \in \Box_B^1$ to the $(m+1+k)$-tuple $(\zvec_m,a,\zvec_k)$. Furthermore, our assumption that $\phi$ is active implies that $\phi(\zvec_m)_i = 0$ for all $1 \leq i \leq n$ by \cref{active-initial-terminal}. Applying \cref{N-explicit}, we can therefore see that the map obtained by pre-composing $N_k(\phi)$ with this critical edge sends $a \in \Box_B^1$ to:
\begin{align*}
(\phi(\zvec_m)_1,\ldots,\phi(\zvec_m)_{n-1}, \phi(\zvec_m)_n \wedge a, \zvec_k) & = (0,\ldots,0,0 \wedge a,0,\ldots,0) \\
& = (0,\ldots,0,0,0,\ldots,0) \\
& = \zvec_{n+k} \\
\end{align*}
So this composite is constant at $\zvec_{n+k}$; thus it factors through the projection $\Box_B^1 \to \Box_B^0$.
\end{proof}

We next consider certain special cases in which $N_k(\phi)$ is degenerate; both of these results follow easily from \cref{N-explicit}.

\begin{lem}\label{N-right-tensor}
For any map $\phi \colon \Box_B^m \to \Box_B^n$, we have $N_k(\phi \otimes \Box_B^1) = (\phi \otimes \Box_B^1 \otimes \Box_B^k) \gamma_{m+1,1}$. \qed
\end{lem}

\begin{lem}\label{N-N}
For all maps $\phi \colon \Box_B^m \to \Box_B^n$ and $k \geq 0$, $N_k(N_0(\phi)) = (N_0(\phi) \otimes \Box_B^k)\gamma_{m+1,1}$. \qed
\end{lem}

\subsection{Decomposition-closed subcomplexes} \label{sec N-closed}

Our next goal is to prove a technical result involving the standard decomposition cubes, which we will specialize in different ways to obtain our main results on the unit of the adjunction $i_! \adjoint i^*$ and the cartesian product of cubical sets.

Recall that by \cref{active-face-factor}, an arbitrary morphism $\psi \colon \Box^m_B \to \Box^n_B$ can be factored uniquely as an active map $\phi \colon \Box^m_B \to \Box^p_B$ followed by a (possibly empty) composite of face maps $\kappa \colon \Box^p_B \to \Box^n_B$. Such factorizations, and their interactions with the $N$ construction, will play a key role in the proofs of this section. As such, we define the following properties of morphisms in $\Box_B$.

\begin{Def} \label{base-dim-def}
    Given a morphism $\kappa \phi$ in $\Box_B$, where $\phi \colon \Box^m_B \to \Box^p_B$ is active and $\kappa \colon \Box^p_B \to \Box^n_B$ is a composite of face maps:
    \begin{itemize}
        \item the \emph{base dimension} of $\kappa \phi$ is the value $p$;
        \item the \emph{tail length} of $\kappa \phi$ is the maximal value of $k$ such that $\phi = \phi' \otimes \Box_B^k$ for some $\phi' \colon \Box_B^{m-k} \to \Box_B^{p-k}$.
    \end{itemize}
\end{Def}

The following is an immediate consequence of \cref{N-explicit}.

\begin{lem}\label{N-tail-length}
    For any active map $\phi$ and $k \geq 0$, the tail length of $N_k(\phi)$ is $k$. \qed
\end{lem}

We next define this section's central objects of study.

\begin{Def}
    A \emph{decomposition-closed subcomplex} of $i^* \Box^n_B$ is a subcomplex $X \subseteq i^* \Box^n_B$ containing the image of the unit map $\Box^n_A \to i^* \Box^n_B$ , such that if a map $\kappa (\phi \otimes \Box^k_B) \colon \Box^m_B \otimes \Box^k_B \to \Box^p_B \to \Box^n_B$, where $\phi \colon \Box^m_B \to \Box^{p-k}_B$ is active and $\kappa \colon \Box^p_B \to \Box^n_B$ is a (possibly empty) composite of face maps, is contained in $X$ when viewed as an $(m+k)$-cube of $i^* \Box^n_B$, then so is the $(m+1+k)$-cube $\kappa N_k(\phi)$.
\end{Def}

For a given $n$, the most natural examples of decomposition-closed subcomplexes are $\Box^n_A$ and $i^* \Box^n_B$ itself. In \cref{cartesian-N-closed} we will see that if $\Sigma, \delta \in B$, then for $m, n \geq 0$, the cartesian product $\Box^m_A \times \Box^n_A$ can also be viewed as a decomposition-closed subcomplex of $i^* \Box^{m+n}_B$.

We can now state the main result of this section.

\begin{prop} \label{N-closed-anodyne}
    Given a pair of decomposition-closed subcomplexes $X \subseteq Y \subseteq i^* \Box^n_B$, the inclusion $X \hookrightarrow Y$ is anodyne. Moreover, if $B \subseteq \FP$, then the inclusion is inner anodyne.
\end{prop}

Our strategy for proving this result will be to construct the cubes of $Y$ from those of $X$ via repeated filling of open boxes, using the property of decomposition-closure and the combinatorial results of \cref{sec N-construction}. To organize this construction, we factor the inclusion $X \hookrightarrow Y$ through a series of intermediate subcomplexes. Thus we now fix $n \geq 0$ and a pair of decomposition-closed subcomplexes $X \subseteq Y \subseteq i^* \Box^n_B$. Our factorization of the inclusion $X \hookrightarrow Y$ will proceed in three steps: we will first filter the cubes of $Y$ based on their base dimension, then their dimension, and finally their tail length. As we will be viewing maps in $\Box_B$ as cubes of $i^* \Box^n_B$, in light of \cref{i-shriek-nondegen} we will refer to such maps as ``degenerate'' if they factor on the right through a degeneracy in $\Box_A$, even though $\Box_B$ itself may not admit an Eilenberg-Zilber structure. 

As the first step of our decomposition, for $0 \leq i \leq n$ we define $W^i$ to be the subcomplex of $Y$ consisting of all cubes of $X$, together with all cubes of $Y$ having base dimension less than or equal to $i$. Thus $W^0 = X$, since the only cubes having base dimension $0$ are the degeneracies of $0$-cubes, all of which are contained in $\Box_A^n$. Similarly, we have $W^n = Y$, since $n$ is the largest base dimension which a map with codomain $\Box^n_B$ can have. We may note that each $W^i$ is still a decomposition-closed subcomplex of $i^* \Box^n_B$, since for any map of the form $\kappa (\phi \otimes \Box^k_B)$ with $\phi$ active and $\kappa$ a composite of face maps, the base dimension of $\kappa N_k(\phi)$ is the same as that of $\kappa (\phi \otimes \Box^k_B)$.

We thus obtain a sequence of inclusions:
\[
X = W^0 \hookrightarrow W^1 \hookrightarrow \cdots \hookrightarrow W^{n-1} \hookrightarrow W^n = Y
\]
To prove \cref{N-closed-anodyne}, therefore, it will suffice to show that each inclusion in the diagram above is (inner) anodyne. To this end, we fix $0 \leq i \leq n - 1$, and focus on the inclusion $W^i \hookrightarrow W^{i+1}$.

As the next step of our factorization, for $j \geq 0$ we let $W^{i,j}$ denote  the subcomplex of $W^{i+1}$ consisting of the following cubes:
\begin{itemize}
    \item all cubes of $W^i$;
    \item all cubes of $Y$ of the form $\kappa \psi$ where $\kappa \colon \Box^{i+1}_B \to \Box^n_B$ is a composite of face maps and $\psi$ is an active map $\Box^{j'}_B \to \Box^{i+1}_B$ for some $j' \leq j$;
    \item all cubes of $Y$ of the form $\kappa N_k(\phi)$ for $\kappa, \psi$ as above where $\psi = \phi \otimes \Box^k_B$ for some active map $\phi \colon \Box^{j'-k}_B \to \Box^{i+1-k}_B$;
    \item degeneracies of the above.
\end{itemize}
To see that this indeed defines a subcomplex of $W^{i+1}$, note that it is closed under degeneracies by definition, and under faces by \cref{N-faces}.

Similarly to the previous step, we see that $W^{i,0} = W^i$ because all $0$-cubes of $i^* \Box^n_B$ are contained in $\Box^n_A$. Thus we obtain an infinite sequence of inclusions:
\[
W^i = W^{i,0} \hookrightarrow W^{i,1} \hookrightarrow \cdots \hookrightarrow W^{i,j} \hookrightarrow W^{i,j+1} \hookrightarrow \cdots
\]
Since every cube of $W^{i+1}$ is contained in some subcomplex $W^{i,j}$, the union of the $W^{i,j}$, ie~the colimit of the diagram above, is $W^{i+1}$. Therefore, because (inner) anodyne maps are closed under transfinite composition, to prove \cref{N-closed-anodyne} it will suffice to prove that each inclusion $W^{i,j} \hookrightarrow W^{i,j+1}$ is inner anodyne. Therefore, we now fix some $j \geq 0$.

As the final step of our decomposition, for $0 \leq k \leq j + 1$, we define $W^{i,j}_k$ to be the subcomplex of $W^{i,j+1}$ consisting of:

\begin{itemize}
    \item all cubes of $W^{i,j}$;
    \item all cubes of $Y$ of the form $\kappa (\phi \otimes \Box^{k'}_B)$, where $\kappa \colon \Box^{i+1}_B \to \Box^n_B$ is a composite of face maps and $\phi \colon \Box^{j+1-k'}_B \to \Box^{i+1-k'}_B$ is active, for $k' \geq k$;
    \item all cubes of $Y$ of the form $\kappa N_{k'}(\phi)$ for $\kappa$, $\phi$ as above;
    \item degeneracies of the above.
\end{itemize}
Once again, to see that this is a subcomplex of $W^{i,j+1}$ we note that it is closed under degeneracies by construction, and can be straightforwardly shown to be closed under faces using \cref{N-faces}.

Similarly to the previous two factorizations, we may note that $W^{i,j}_{j+1} = W^{i,j}$, since the only active map $\phi$ with domain $\Box^{0}_B$ is the identity, and for this map and any composite of faces $\kappa$, the maps $\kappa (\phi \otimes \Box^{j+1}_B)$ and $\kappa N_k(\phi)$ are in $\Box_A$. Likewise, $W^{i,j}_0 = W^{i,j+1}$, and we have a sequence of inclusions:
\[
W^{i,j} = W^{i,j}_{j+1} \hookrightarrow W^{i,j}_j \hookrightarrow \cdots \hookrightarrow W^{i,j}_1 \hookrightarrow W^{i,j}_0 = W^{i,j+1}
\]

We have thus reduced our task to proving that each of the inclusions in the diagram above is (inner) anodyne; it is this task which we will take up by means of open-box filling. To this end, in addition to the values of $n, i, j$ already fixed, let us also fix $k$ satisfying $0 \leq k \leq j$.

 Our task will be to analyze the inclusion $W^{i,j}_{k+1} \hookrightarrow W^{i,j}_k$. We may view this inclusion as adjoining to $W^{i,j}_{k+1}$ all cubes of $Y$ of the form $\kappa (\phi \otimes \Box^k_B)$ where $\kappa \colon \Box^{i+1}_B \to \Box^n_B$ is a composite of face maps and $\phi \colon \Box^{j+1-k}_B \to \Box^{i+1-k}_B$ is active, 
as well as the corresponding cubes obtained by replacing $\phi \otimes \Box^k_B$ with $N_k(\phi)$, and degeneracies of these. Note, however, that some cubes of this form are already contained in $W^{i,j}_{k+1}$, such as those for which $\phi = \phi' \otimes \Box^1$ for some map $\phi'$ (as in that case, $\phi \otimes \Box^k_B = \phi' \otimes \Box^{k+1}_B$). Thus we begin by determining which new cubes must be adjoined to $W^{i,j}_{k+1}$ in order to construct $W^{i,j}_k$.

\begin{lem}\label{W-i-j-k-redundancy}
A cube $\kappa (\phi \otimes \Box^k_B)$, where $\kappa \colon \Box^{i+1}_B \to \Box^n_B$ is a composite of face maps and $\phi \colon \Box^{j+1-k}_B \to \Box^{i+1-k}_B$ is active, is present in $W^{i,j}_{k+1}$ if and only if one of the following cases holds:
\begin{itemize}
\item $\kappa (\phi \otimes \Box^k_B)$ is a cube of $X$;
\item $\phi \otimes \Box^k_B$ is degenerate;
\item $\phi \otimes \Box^k_B = \phi' \otimes \Box^{k+1}_B$ for some active map $\phi' \colon \Box^{j-k}_B \to \Box^{i-k}_B$;
\item $\phi \otimes \Box^k_B = N_p(\phi')$ for some map $\phi' \colon \Box^{j-p}_B \to \Box^{i+1-p}_B$.
\end{itemize}
Moreover, in each of these cases $\kappa N_k(\phi)$ is present in $W^{i,j}_{k+1}$ as well.
\end{lem}

\begin{proof}
That $\kappa (\phi \otimes \Box^k_B)$ is present in $W^{i,j}_{k+1}$ if and only if one of the listed cases holds is immediate from the definition of $W^{i,j}_{k+1}$. (Note that $\kappa (\phi \otimes \Box^k_B)$ is degenerate if and only if $\phi \otimes \Box^k_B$ is.) To show that $N_k(\phi)$ is present in $W^{m,k+1}$ in each of these cases, we consider each case in turn.

\begin{itemize}
\item Suppose $\kappa (\phi \otimes \Box^k_B)$ is a cube of $X$. Then because $X$ is decomposition-closed by assumption, $\kappa N_k(\phi)$ is contained in $X$, and thus in $W^{i,j}_{k+1}$.
\item Suppose $\phi \otimes \Box^k_B$ is degenerate; this implies that $\phi$ is a degeneracy of some  $\phi' \colon \Box_B^{j-k} \to \Box_B^{i+1-k}$.
That $\phi'$ is active is a straightforward consequence of \cref{active-face-factor}. By \cref{N-identities}, $N_k(\phi)$ is a degeneracy of $N_k(\phi')$. Since $\kappa N_k(\phi')$ is a cube of $W^{i,j} \subseteq W^{i,j}_{k+1}$, so is $N_k(\phi)$.
\item Suppose $\phi \otimes \Box_B^k = \phi' \otimes \Box_B^{k+1}$ for some active map $\phi' \colon \Box_B^{j-k} \to \Box_B^{i-k}$. Then by \cref{tensor-faithful},  we have $\phi = \phi' \otimes \Box_B^1$. (In other words, this is precisely the case in which $\phi$ has positive tail length.) Thus \cref{N-right-tensor} implies that $\kappa N_{k}(\phi)$ is a degeneracy of $\kappa (\phi \otimes \Box_B^k)$, hence is also contained in $W^{i,j}_{k+1}$.
\item Finally, suppose that none of the previous cases hold; then the only remaining case in which $\kappa(\phi \otimes \Box^k_B)$ could be contained in $W^{i,j}_{k+1}$ is if $\phi \otimes \Box_B^k = N_p(\phi')$ for some map $\phi' \colon \Box_B^{j-p} \to \Box_B^{i+1-p}$. Because the previous case in particular does not hold, $\phi$ has tail length $0$, so that $\phi \otimes \Box^k_B$ has tail length $k$. By \cref{N-tail-length}, it must therefore be the case that $p = k$. By \cref{tensor-faithful,N-k-tensor}, therefore, we have $\phi = N_0(\phi')$. Therefore, by \cref{N-N}, $\kappa N_k(\phi)$ is a degeneracy of $\kappa (\phi \otimes \Box_B^k)$, and hence is also contained in $W^{i,j}_{k+1}$. \qedhere
\end{itemize}
\end{proof}

Thus the cubes of $W^{i,j}_k$ which we must construct by open box filling are those falling under none of the listed cases above. We next show that the necessary open boxes do indeed exist in $W^{i,j}_{k+1}$.

\begin{lem}\label{W-i-j-k-open-boxes}
For an active map $\phi \colon \Box^{j+1-k}_B \to \Box^{i+1-k}_B$ and composite of faces $\kappa \colon \Box^{i+1}_B \to \Box^n_B$ such that $\kappa (\phi \otimes \Box^k_B)$ is not present in $W^{i,j}_{k+1}$, all faces of $\kappa N_k(\phi)$ aside from its $(j+2-k,1)$-face are present in $W^{i,j}_{k+1}$.
\end{lem}

\begin{proof}
We analyze the faces of $\kappa N_k(\phi)$ using \cref{N-faces}. We first consider a face $\kappa N_k(\phi)\bd_{p,\varepsilon}$ with $1 \leq p \leq j + 1 - k$; here \cref{N-faces} provides three possible cases.

In case \ref{face-case-bdry}, we have $N_k(\phi)\bd_{p,\varepsilon} = \bd_{q,\varepsilon'} \psi$ for some face map $\bd_{q,\varepsilon'} \colon \Box^i_B \to \Box^{i+1}_B$ and some $\psi \colon \Box^{j+1}_B \to \Box^i_B$. It follows that $\kappa N_k(\phi) \bd_{p,\varepsilon} = \kappa \bd_{q,\varepsilon'} \psi$ has base dimension at most $i$, and is therefore contained in $W^i \subseteq W^{i,j}_{k+1}$.

In case \ref{face-case-tensor}, we have $N_k(\phi)\bd_{p,\varepsilon} = \psi \otimes \Box^{k+1}_B$ for an active map $\psi \colon \Box^{j-k}_B \to \Box^{i-k}_B$, so $\kappa N_k(\phi)\bd_{p,\varepsilon} = \kappa (\psi \otimes \Box^{k+1}_B)$ is contained in $W^{i,j}_{k+1}$ by definition.

In case \ref{face-case-N}, $N_k(\phi) \bd_{i,\varepsilon} = N_k(\psi)$ for an active map $\psi \colon \Box^{j-k}_B \to \Box^{i+1-k}_B$, so $\kappa N_k(\phi)\bd_{i,\varepsilon} = \kappa N_k(\psi)$ is contained in $W^{i,j} \subseteq W^{i,j}_{k+1}$.

Finally, for faces $N_k(\phi) \bd_{p,\varepsilon}$ with $(p,\varepsilon) = (j+2-k,0)$ or $p > j+2-k$, \cref{N-faces} shows that $N_k(\phi)\bd_{p,\varepsilon}$ factors through a face map on the left, so once again $\kappa N_k(\phi) \bd_{p,\varepsilon}$ is contained in $W^i \subseteq W^{i,j}_{k+1}$.
\end{proof}

We have now established sufficient results about decomposition-closed subcomplexes to prove the main result of this subsection.

\begin{proof}[Proof of \cref{N-closed-anodyne}]
We first note that by the closure of (inner) anodyne maps under composition and transfinite composition, it suffices to show that each inclusion $W^{i,j}_{k+1} \hookrightarrow W^{i,j}_k$ is (inner) anodyne.

By the definition of $W^{i,j}_k$ and \cref{W-i-j-k-redundancy}, constructing $W^{i,j}_k$ from $W^{i,j}_{k+1}$ amounts to adjoining to $W^{i,j}_{k+1}$ the cubes $\kappa (\phi \otimes \Box_B^k)$ and $\kappa N_k(\phi)$ for all active maps $\phi \colon \Box_B^{i+1-k} \to \Box_B^{j+1-k}$ not covered by any of the cases listed in the statement of that result. Moreover, we may note that in these cases, $\kappa (\phi \otimes \Box_B^k)$ and $\kappa N_k(\phi)$ are non-degenerate as cubes of $i^* \Box^n_B$; to see this, we may note that $\kappa (\phi \otimes \Box_B^k)$ is non-degenerate by assumption, and that by \cref{N-faces,W-i-j-k-open-boxes} it appears as exactly one face of $\kappa N_k(\phi)$. Thus we can apply \cref{degen-one-face} to show that $\kappa N_k(\phi)$ is non-degenerate as well. 

By \cref{W-i-j-k-open-boxes}, for each such $\kappa, \phi$,the faces of $\kappa N_k(\phi)$ other than its $(j+2-k,1)$-face $\kappa \phi \otimes \Box^k_B$ form a $(j+2-k,1)$-open box in $W^{i,j}_{k+1}$. Thus we may obtain $W^{i,j}_k$ from $W^{i,j}_{k+1}$ by filling all of these open boxes; in other words, we have a pushout diagram:
\[
\begin{tikzcd}
\coprod \sqcap^{j+2}_{A,j+2-k,1} \ar[r] \ar[d] \pushout & W^{i,j}_{k+1} \ar[d] \\
\coprod \Box^{j+2}_A \ar[r] & W^{i,j}_k \\
\end{tikzcd}
\]
where the coproduct is taken over all composites of face maps $\kappa \colon \Box_B^{i+1} \to \Box_B^n$ and active maps $\phi \colon \Box_B^{j+1-k} \to \Box_B^{i+1-k}$ such that $\kappa (\phi \otimes \Box_B^k)$ is not contained in $W^{i,j}_{k+1}$. For such a pair $(\kappa,\phi)$, the $(\kappa,\phi)$-component of the bottom horizontal map picks out the cube $\kappa N_k(\phi)$, while that of the top horizontal map picks out the $(j+2-k,1)$-open box on $\kappa N_k(\phi)$. Thus $W^{i,j}_{k+1} \hookrightarrow W^{i,j}_k$ is anodyne. 

Moreover, by \cref{N-crit-edge}, if $B \subseteq \FP$ then these open boxes are inner. Thus we could replace the pushout diagram above with one in which the left vertical map is a coproduct of inclusions $\hcap^{j+2}_{A,j+2-k,1} \hookrightarrow \hBox^{j+2}_{A,j+2-k,1}$, thereby showing that $W^{i,j}_{k+1} \hookrightarrow W^{i,j}_k$ is inner anodyne.
\end{proof}

\section{The unit of the left Kan extension adjunction} \label{sec unit}

We now consider a generalization of \cref{unit-tcof-empty} which follows straightforwardly from \cref{N-closed-anodyne}.

\begin{prop}\label{unit-tcof}
For $A \subseteq B \subseteq \FS$, with $A \subseteq \conset$ and $B$ containing at least one element of $\conset$, the unit $\eta \colon \id_{\cSet_A} \Rightarrow i^* i_!$ of the adjunction $i_! \adjoint i^*$ induced by the inclusion $i \colon \Box_A \hookrightarrow \Box_B$ is a natural trivial cofibration in the Grothendieck model structure on $\cSet_A$. Moreover, if $B \subseteq \FP$, then $\eta$ is a natural trivial cofibration in the cubical Joyal model structure on $\cSet_A$.
\end{prop}

\begin{proof}
\cref{unit-mono} shows that the unit of $i_! \adjoint i^*$ is a cofibration, thus it remains only to prove that it is a weak equivalence. The case $B \subseteq \conset$ is given by \cref{unit-tcof-empty}. For our remaining cases, we first note that $i^* i_!$ preserves pushouts as a left adjoint, and preserves monomorphisms by \cref{i-shriek-monos} and the fact that $i^*$ is a right adjoint. Therefore, by a standard induction on skeleta argument involving the gluing lemma (cf~\cite[Lem.~1.6]{doherty:without-connections}), it will suffice to consider the components of the unit at representable cubical sets.

First suppose $\wedge \in A$; note that since $A \subseteq B$, in this case the condition that $B$ contains at least one of $\wedge, \vee$ is satisfied automatically. In this case, the statement that each map $\Box^n_A \hookrightarrow i^*\Box^n_B$ is a trivial cofibration follows from \cref{N-closed-anodyne}, taking $X = \Box^n_A, Y = i^* \Box^n_B$.
For the case $\vee \in A$, we can prove the claim by a similar argument, where the cubes $N_k(\phi)$ are constructed using negative connections. 

Finally, we consider the case in which $A = \varnothing$ and $B$ is not assumed to be a subset of $\conset$. For concreteness, assume $\wedge \in B$; the case $\vee \in B$ is analogous.

In this case, we factor the inclusion $\Box_{\varnothing} \hookrightarrow \Box_B$ as a composite of inclusions $j \colon \Box_\varnothing \hookrightarrow \Box_\wedge$ and $l \colon \Box_\wedge \hookrightarrow \Box_B$. Then the adjunction $i_! \adjoint i^*$ is the composite of the adjunctions $j_! \adjoint j^*$ and $l_! \adjoint l^*$. Denote the unit of $j_! \adjoint j^*$ by $\eta^j$ and the unit of $l_! \adjoint l^*$ by $\eta^l$. Then the component of the unit of $i_! \adjoint i^*$ at an object $X \in \cSet_\varnothing$ is the composite
\[
\begin{tikzcd}
X \ar[r,hookrightarrow,"\eta^j_X"] & j^* j_! X \ar[r,hookrightarrow,"j^* \eta^l_{j_! X}"] & j^* l^* l_! j_! X \cong i^* i_! X
\end{tikzcd}
\]
We have shown that $\eta^j$ is a natural weak equivalence in $\cSet_\varnothing$, and that $\eta^l$ is a natural weak equivalence in $\cSet_\wedge$. Moreover, $j^*$ preserves weak equivalences by \cref{Quillen-equiv-empty}. Thus the composite depicted above is a  weak equivalence.
\end{proof}

\begin{rmk}
    It is natural to wonder whether \cref{unit-tcof} can be generalized -- for instance, whether the unit of the adjunction $i_! \adjoint i^*$ induced by a proper inclusion $\Box_A \hookrightarrow \Box_B$ can be a trivial cofibration in the cubical Joyal model structure on $\cSet_A$ if $\Box_B$ contains reversals, or in either the cubical Joyal or the Grothendieck model structure if $A = \varnothing$ and $\Box_B$ does not contain connections. In fact, neither of these results holds; in all cases it can be verified that the component of the unit at a representable cube of dimension 1 or 2 is not a weak equivalence.
\end{rmk}

The remainder of this section will be devoted to exploring the consequences of \cref{unit-tcof}. We first apply it to show that although the monoidal structure given by the geometric product on $\cSet_A$ for $A \subseteq \conset$ is not symmetric, it is ``symmetric up to natural weak equivalence'' in the cubical Joyal model structure.

\begin{thm}\label{monoidal-weak-equiv}
For $A \subseteq \conset$ and $X, Y \in \cSet_A$, we have a zigzag of weak equivalences in the cubical Joyal model structure relating $X \otimes Y$ and $Y \otimes X$, natural in $X$ and $Y$.
\end{thm}

\begin{proof}
Again, let $B \subseteq \FP$ such that $A \subseteq B$ and $B$ contains $\Sigma$ and at least one of $\wedge, \vee$, and consider $i \colon \Box_A \hookrightarrow \Box_B$. By \cref{unit-tcof} we have a natural trivial cofibration $X \otimes Y \hookrightarrow i^* i_! (X \otimes Y)$. The codomain of this map is isomorphic to $i^*(i_! X \otimes i_! Y)$ by \cref{i-shriek-monoidal}; this, in turn, is isomorphic to $i^* (i_! Y \otimes i_! X)$ since the geometric product in $\cSet_B$ is symmetric. A further application of \cref{i-shriek-monoidal} shows that this is naturally isomorphic to $i^* i_! (Y \otimes X)$, and a further application of \cref{unit-tcof} provides a natural trivial cofibration $Y \otimes X \hookrightarrow i^* i_! (Y \otimes X)$.
\end{proof}

We now consider the use of \cref{unit-tcof} to induce model structures on categories $\cSet_A$ where $A$ contains at least one kind of connection, but is not a subset of $\conset$. In particular,  we will thus obtain models for $(\infty,1)$-categories using some kinds of cubical sets with symmetries.

We begin by reviewing the definition of an induced model structure.

\begin{Def}
Let $F : \mathsf{C} \rightleftarrows \mathsf{D} : U$ be an adjunction between model categories. The model structure on $\mathsf{C}$ is \emph{left induced} by $F$ if $F$ creates cofibrations and weak equivalences. Likewise, the model structure on $\mathsf{D}$ is \emph{right induced} by $U$ if $U$ creates fibrations and weak equivalences.
\end{Def}

Note that for a given adjunction $\mathsf{C} \rightleftarrows \mathsf{D}$ and a given model structure on $\mathsf{D}$, the left-induced model structure on $\mathsf{C}$ is unique, if one exists, since the definition determines its cofibrations and weak equivalences. Likewise, for a given model structure on $\mathsf{C}$, the right-induced model structure on $\mathsf{D}$ is unique, if one exists.

Our constructions will follow from the application of established results on the existence of induced model structures, enabled by \cref{unit-tcof}. For ease of reference, we package the model-categorical reasoning involved into the following general result.

\begin{prop}\label{induced-construction-general}
    Let $\mathsf{C}$ be a combinatorial model category, $\mathsf{D}$ a locally presentable category, and $F \colon \mathsf{D} \to \mathsf{C}$ a functor admitting a left adjoint $L$ and a right adjoint $R$. Suppose that the composite functor $FL$ preserves cofibrations, and the unit $\id_{\mathsf{C}} \Rightarrow FL$ is a natural weak equivalence. Then $\mathsf{D}$ admits both a left-induced model structure $\mathsf{D}_l$ and a right-induced model structure $\mathsf{D}_r$. Moreover:
    \begin{itemize}
        \item the adjunctions $L \adjoint F$ and $F \adjoint R$ are Quillen with respect to both $\mathsf{D}_r$ and $\mathsf{D}_l$, with $L \adjoint F$ being a Quillen equivalence in both cases;
        \item the adjunction $\id_{\mathsf{D}} : \mathsf{D}_r \rightleftarrows \mathsf{D}_l : \id_{\mathsf{D}}$ is a Quillen equivalence.
    \end{itemize}
\end{prop}

\begin{proof}
    The composite functor $FL$ preserves cofibrations by assumption, and preserves weak equivalences by the assumption that $\id_{\mathsf{C}} \Rightarrow FL$ is a natural weak equivalence together with two-out-of-three. Thus $FL$ is a left Quillen functor; the existence of the right-induced model structure, and the fact that $L \adjoint F$ and $F \adjoint R$ are Quillen adjunctions with respect to this model structure, follow by \cite[Thm.~2.3]{drummond-cole-hackney:right-induced}. To see that $L \adjoint F$ is a Quillen equivalence between $\mathsf{C}$ and $\mathsf{D}_r$, we note that the definition of a right-induced model structure and our assumption on the unit of $L \adjoint F$ imply that hypothesis (c) of \cite[Cor.~1.3.16]{hovey:book} is satisfied.

    Next we show that the left-induced model structure $\mathsf{D}_l$ exists and that $L \adjoint F$ and $F \adjoint R$ are Quillen adjunctions with respect to this model structure as well. For this, we note that our assumption on the unit of $L \adjoint F$, together with \cite[Rmk.~1.11]{hackney-rovelli:left-induced}, implies that $L \adjoint F \adjoint R$ is a homotopy idempotent string in the sense of \cite[Def.~1.10]{hackney-rovelli:left-induced}; this, together with our previous observation that $FL$ is left Quillen, allows us to apply \cite[Thm.~1.13]{hackney-rovelli:left-induced}.

    That $\id_{\mathsf{D}} : \mathsf{D}_r \rightleftarrows \mathsf{D}_l : \id_{\mathsf{D}}$ is a Quillen equivalence follows from \cite[Prop.~1.16]{hackney-rovelli:left-induced}. We then see that $L : \mathsf{C} \rightleftarrows \mathsf{D}_l : F$ is a Quillen equivalence as the composite of $L : \mathsf{C} \rightleftarrows \mathsf{D}_r : F$ with $\id_{\mathsf{D}} : \mathsf{D}_r \rightleftarrows \mathsf{D}_l : \id_{\mathsf{D}}$.
\end{proof}

In the case of right-induced model structures, generating sets of cofibrations, trivial cofibrations, and anodyne maps can easily be obtained from those of the original model category via the following well-known result.

\begin{prop}[{cf.~\cite[Thm.~11.3.2]{hirschhorn:book}}]\label{right-gen-cofs}
Let $F : \catC \rightleftarrows \catD : U$ be an adjunction, with $\catC$ a model category, and let $I$ be a class of generating cofibrations (resp.~generating trivial cofibrations, pseudo-generating trivial cofibrations) for $\catC$. If the right-induced model structure on $\catD$ exists, then $FI$ is a class of generating cofibrations (resp.~generating trivial cofibrations, pseudo-generating trivial cofibrations) for this model structure.
\end{prop}

\begin{proof}
We consider the case of generating cofibrations; the other cases are similar. A morphism $f$ in $\catD$ is a trivial fibration in the right-induced model structure if and only if $Uf$ has the right lifting property with respect to all maps of $I$. This, in turn, holds if and only if $f$ has the right lifting property with respect to all maps of $FI$.
\end{proof}

We now construct the desired model structures on cubical sets.

\begin{thm}\label{induced-model-structures}
Let $\Box_A$ be a cube category containing at least one kind of connection. Then $\cSet_A$ admits the following model structures:
\begin{itemize}
\item the \emph{right-induced Grothendieck model structure}, in which weak equivalences and fibrations are created by $i^* \colon \cSet_A \to \cSet_\varnothing$ from the Grothendieck model structure on $\cSet_\varnothing$;
\item the \emph{left-induced Grothendieck model structure}, in which weak equivalences and cofibrations are created by $i^* \colon \cSet_A \to \cSet_\varnothing$ from the Grothendieck model structure on $\cSet_\varnothing$.
\end{itemize}
Moreover, if $\Box_A$ does not contain reversals, then $\cSet_A$ admits the following model structures:
\begin{itemize}
\item the \emph{right-induced cubical Joyal model structure}, in which weak equivalences and fibrations are created by $i^* \colon \cSet_A \to \cSet_\varnothing$ from the cubical Joyal model structure on $\cSet_\varnothing$;
\item the \emph{left-induced cubical Joyal model structure}, in which weak equivalences and cofibrations are created by $i^* \colon \cSet_A \to \cSet_\varnothing$ from the cubical Joyal model structure on $\cSet_\varnothing$.
\end{itemize}
Furthermore:
\begin{itemize}
    \item for each induced model structure, the adjunction $i_! : \cSet_\varnothing \rightleftarrows \cSet_A : i^*$ is a Quillen equivalence;
    \item for each induced model structure, the adjunction $i^* : \cSet_A \rightleftarrows \cSet_\varnothing : i_*$ is a Quillen adjunction;
    \item in each case, the identity adjunction on $\cSet_A$ is a Quillen equivalence between the left- and right-induced model structures, with the right-induced model structure as the domain of the left adjoint.
\end{itemize}
\end{thm}

\begin{proof}
We verify the hypotheses of \cref{induced-construction-general} with respect to $i^* \colon \cSet_A \to \cSet_\varnothing$. The composite functor $i^* i_!$ preserves cofibrations by \cref{i-shriek-monos} and the fact that $i^*$ preserves monomorphisms as a right adjoint, while the unit of $i_! \adjoint i^*$ is a natural weak equivalence by \cref{unit-tcof}. 
\end{proof}

We devote the remainder of this section to analysis of the model structures of \cref{induced-model-structures}.

\begin{Def}[{cf~\cite[Defs.~8.1.23 \& 8.1.30]{CisinskiAsterisque}}]\label{normal-mono-def}
    A monomorphism $X \to Y$ in a category of cubical sets $\cSet_A$ is \emph{normal} if for every $n \geq 0$, the automorphism group of $\Box^n_A$ acts freely on the non-degenerate elements of $Y_n \setminus X_n$. 

    A cubical set $X \in \cSet_A$ is \emph{normal} if $\varnothing \hookrightarrow X$ is a normal monomorphism, ie, if the automorphism group of each $\Box^n_A$ acts freely on the non-degenerate cubes of $X$.
\end{Def}

In particular, we may note that if $\Box_A$ does not contain symmetries or reversals, then all automorphism groups of objects of $\Box_A$ are trivial, so that all monomorphisms of $\cSet_A$ are normal.

\begin{prop}\label{induced-cofs}
The left-induced model structures of \cref{induced-model-structures} have monomorphisms as their cofibrations. In the right-induced model structures of \cref{induced-model-structures}, the cofibrations are generated by the set of boundary inclusions. In particular, if $A$ does not contain diagonals, then the cofibrations of the right-induced model structures are the normal monomorphisms.
\end{prop}

\begin{proof}
For the left-induced model structures, we note that by the definition of $i^*$ it is immediate that a map $f$ in $\cSet_A$ is a monomorphism if and only if $i^* f$ is a monomorphism.

The characterization of the generating cofibrations in the right-induced model structure is immediate from \cref{i-shriek-open-box,right-gen-cofs}. In the case of cubical sets without diagonals, it follows that the cofibrations are the normal monomorphisms by \cite[Prop.~8.1.35]{CisinskiAsterisque} and \cite[Thm.~7.9]{campion:EZ-cubes}.
\end{proof}

\begin{cor}\label{right-induced-cof-obs}
   For $A$ as in the statement of \cref{induced-model-structures}, if $\delta \notin A$, then the cofibrant objects of both right-induced model structures on $\cSet_A$ are the normal cubical sets. \qed
\end{cor}

\begin{prop}\label{induced-gen-tcofs}
    In the right-induced Grothendieck model structures, the open box inclusions form a generating set of trivial cofibrations. In the right-induced cubical Joyal model structures, the inner open box inclusions and endpoint inclusions into $K$ form a set of pseudo-generating trivial cofibrations.
\end{prop}

\begin{proof}
    As in the proof of \cref{induced-cofs}, this is immediate from \cref{i-shriek-open-box,right-gen-cofs}.
\end{proof}

\begin{rmk}\label{i-induction-consistency}
For any $A \subseteq B \subseteq \FS$ with $A \subseteq \conset$, we could instead construct the induced model structures on $\cSet_B$ of \cref{induced-model-structures} using $i^* \colon \cSet_B \to \cSet_A$. To see that the model structure thus produced is independent of $A$, we may first note that by \cref{Quillen-equiv-empty}, the class of weak equivalences in $\cSet_B$ created by $i^* \colon \cSet_B \to \cSet_A$ coincides with that created by $i^* \colon \cSet_B \to \cSet_\varnothing$. 
We could then apply the argument of \cref{induced-cofs} verbatim to obtain identical characterizations of both model structures' cofibrations.
\end{rmk}

Note that all of the cube categories under consideration here, except for those which contain diagonals but not symmetries, are shown to be test categories in \cite[Cor.~3]{buchholtz-morehouse:varieties-of-cubes}. With this in mind, we may obtain an alternative characterization of the left-induced Grothendieck model structures of \cref{induced-model-structures} in these cases. 

\begin{prop} \label{left-induced-test}
Let $\Box_A$ be a cube category containing at least one kind of connection which is also a test category. Then the left-induced Grothendieck model structure of \cref{induced-model-structures} coincides with the test model structure on $\cSet_A$.
\end{prop}

\begin{proof}
Both the test model structure and the left-induced Grothendieck model structure have monomorphisms as their cofibrations (the latter by \cref{induced-cofs}), so it suffices to show that the weak equivalences of the test model structure are created by $i^* \colon \Box_{A} \to \Box_{\varnothing}$. For this, by \cite[Thm.~4.2.23]{CisinskiAsterisque} (implication $(b'') \Rightarrow (d)$) it suffices to show that $i^* \Box^n_A$ is contractible in the Grothendieck model structure on $\cSet_{\varnothing}$ for all $n$.  This, in turn, follows from \cref{unit-tcof} and the contractibility of $\Box^n_\varnothing$ in the Grothendieck model structure.
\end{proof}

\begin{rmk} \label{HR-comparison}
    In the case $A \subseteq \conset$, we have Quillen equivalences $T : \cSet_A \rightleftarrows \sSet : U$ between the cubical Joyal (resp.~Grothendieck) model structure on $\cSet_A$ and the Joyal (resp.~Quillen) model structure on $\sSet$ \cite[Thms.~6.1 \& 6.26]{doherty-kapulkin-lindsey-sattler}. The left adjoint $T$, the \emph{triangulation functor}, which sends each cube $\Box^n_A$ to the simplicial set $(\Delta^1)^n$, creates the weak equivalences of $\cSet_A$ in both cases. Thus, for any inclusion of cube categories $i \colon \Box_A \hookrightarrow \Box_B$ with $A \subseteq \conset$, the weak equivalences of the induced model structures on $\cSet_B$ are created by the composite functor $Ti^* \colon \cSet_B \to \sSet$.

    The triangulation functor can be similarly defined for any cube category $\Box_B$ not containing reversals. For the case $B = \{ \vee, \Sigma, \delta \}$ (and similarly for the isomorphic case $B = \{ \wedge, \Sigma, \delta \}$), \cite[Thm.~7.8 \& Cor.~7.23]{cavallo-sattler:relative-elegance} shows that $T \colon \cSet_B \to \sSet$, where $\sSet$ is equipped with the Quillen model structure, creates the weak equivalences of the test model structure. Thus $T$ creates the same class of weak equivalences as $i^*$ by \cref{left-induced-test}. In the case $B = \FP$, it can similarly be shown that $T$ creates the weak equivalences of the test model structure using the characterization of $T$ given in \cite{sattler:idempotent-completion}.
    
  In general, it is an open question whether $T \colon \cSet_B \to \sSet$ (where $\sSet$ is equipped with the Quillen or Joyal model structure as appropriate) creates the weak equivalences of the induced model structures on $\cSet_B$. In the case $B = \FP$, proving this for the left-induced cubical Joyal model structure would be equivalent to proving that it coincides with that constructed in \cite[Prop.~2.3]{hackney-rovelli:left-induced} under the name $\cSet_{(m,1)}$, which also has monomorphisms as its cofibrations and has weak equivalences created by triangulation. 
\end{rmk}

\section{Cartesian monoidality} \label{sec monoidal}

We now consider the Cartesian product of cubical sets, with the aim of showing that the cubical Joyal model structures on categories $\cSet_A$ with $A$ a non-empty subset of $\conset$ are cartesian monoidal. We will do this using a natural map between the geometric and cartesian products, which we will show to be a trivial cofibration in the cubical Joyal model structure. Our techniques will also allow us to obtain a new proof of cartesian monoidality for the Grothendieck model structures on these categories. (That the Grothendieck model structures are cartesian monoidal was previously known, as a consequence of the fact that in these cases $\Box_A$ is a strict test category -- see \cite[Prop.~4.3]{maltsiniotis:connections-strict-test-cat} as well as \cite[Thm.~1.7]{CisinskiUniverses}.)

We begin with a general result concerning pushout products in model categories.

\begin{lem}\label{pop-axiom-sufficient}
    Let $\mathcal{C}$ be a model category equipped with a bifunctor $\odot \colon \mathcal{C} \times \mathcal{C} \to \mathcal{C}$. Suppose the following are true:
    \begin{itemize}
        \item if $i$ and $j$ are cofibrations in $\mathcal{C}$, then so is the pushout product $i \widehat{\odot} j$;
        \item the bifunctor $\odot$ preserves the cofibrations and trivial cofibrations of $\mathcal{C}$ in each variable.
    \end{itemize}
    Then a pushout product of cofibrations $i \widehat{\odot} j$ is a trivial cofibration if either $i$ or $j$ is trivial.
\end{lem}

\begin{proof}
    Let $i \colon A \to B$ and $j \colon X \to Y$ be a pair of cofibrations in $\mathcal{C}$, and suppose that $j$ is trivial; the case where $i$ is trivial is similar. The pushout product $i \widehat{\odot} j$ is a cofibration by assumption, so we need only show that it is a weak equivalence.

    We first note that $A \odot j \colon A \odot X \to A \odot Y$ is a trivial cofibration by assumption. Thus its pushout $B \odot X \to A \odot Y \cup_{A \odot X} B \odot X$ is a trivial cofibration as well. Likewise, $B \odot j \colon B \odot X \to B \odot Y$ is a trivial cofibration by assumption.

    Now consider the following commuting diagram:
    \[
    \begin{tikzcd}
        B \odot X \ar[r] \ar[dr,swap,"B \odot j"] & A \odot X \cup_{A \odot X} B \odot X \ar[d, "i \widehat{\odot} j"] \\
        & B \odot Y \\
    \end{tikzcd}
    \]
    It follows that $i \widehat{\odot} j$ is a weak equivalence by two-out-of-three.
\end{proof}

\subsection{Comparison of the geometric and cartesian products}

We now proceed to our analysis of the geometric and cartesian products of cubical sets. Throughout this subsection, fix $A \subseteq B \subseteq \FP$, with $\Sigma, \delta \in B$. This implies, in particular, that for each $n \geq 0$ there is a map $D_n \colon \Box^n_B \to \Box^{2n}_B$ sending $(a_1,\ldots,a_n)$ to $(a_1,\ldots,a_n,a_1,\ldots,a_n)$. (For $n = 0$ this is the identity.)

To aid in understanding the maps we will use in our proofs, we define the following composites of degeneracy morphisms for $m, n \geq 0$.
\begin{itemize}
    \item The map $\sigma^F_m \colon \Box^{m+n}_A \to \Box^n_A$ deletes the first $m$ components, sending $(a_1,\ldots,a_m,b_1,\ldots,b_n)$ to $(b_1,\ldots,b_n)$.
    \item The map $\sigma^L_n \colon \Box^{m+n}_A \to \Box^m_N$ deletes the last $n$ components, sending $(a_1,\ldots,a_m,b_1,\ldots,b_n)$ to $(a_1,\ldots,a_m)$.
\end{itemize}

(In particular, both $\sigma^F_0$ and $\sigma^L_0$ are identities.) By straightforward calculations, we can obtain the following lemmas involving the maps defined above.

\begin{lem}\label{long-diagonal-duplicates}
    For any $\phi \colon \Box^m_B \to \Box^n_B$, we have $D_n \phi = (\phi \otimes \phi) D_m$. \qed
\end{lem}

\begin{lem}\label{long-diagonal-retraction}
    For any $m, n \geq 0$, the map $\sigma^L_n \otimes \sigma^F_m \colon \Box^{m+n}_A \otimes \Box^{m+n}_A \to \Box^m_A \otimes \Box^n_A$ is a retraction of $D_{m+n}$. \qed 
\end{lem}

Note that we may take either of the dimension variables in the lemma above to be 0, obtaining the result for any $n$, the maps $\sigma^L_n, \sigma^F_n \colon \Box^{2n}_A \to \Box^n_A$ are retractions of $D_n$.

We begin by defining the comparison map which will be our object of study. For any $A \subseteq \FS$, given $X, Y \in \cSet_A$, we define the projection map $\pi_X$ to be the composite $X \otimes Y \to X \otimes \Box^0_A \cong X$. We similarly define a projection map $\pi_Y \colon X \otimes Y \to Y$.

Then for a pair of cubes $x \colon \Box^m_A \to X, y \colon \Box^n_A \to Y$, the map $\pi_X$ sends the $(m+n)$-cube $x \otimes y$ of $X \otimes Y$ to $x \sigma^L_n$, while $\pi_Y$ sends $x \otimes y$ to $y \sigma^F_m$. In the case of a pair of 1-cubes $x$ and $y$, with initial and terminal vertices $x_0, x_1, y_0, y_1$, we illustrate $x \otimes y$ and its images under the projection maps to $X$ and $Y$ below; these images are the degeneracies $x \sigma^L_1 = x \sigma_2$ and $y \sigma^F_1 = y \sigma_1$, respectively.

\[
\begin{tikzcd}
    x_0 \otimes y_0 \ar[r,"x \otimes y_0"] \ar[d,swap,"x_0 \otimes y"] \ar[dr,phantom,"x \otimes y"] & x_1 \otimes y_0 \ar[d,"x_1 \otimes y"] & x_0 \ar[d,equal] \ar[r,"x"] \ar[dr,phantom,"x \sigma_2"] & x_1 \ar[d, equal] & y_0 \ar[r,equal] \ar[d,swap,"y"] \ar[dr,phantom,"y \sigma_1"] & y_0 \ar[d,"y"] \\
    x_0 \otimes y_1 \ar[r,swap,"x \otimes y_1"] & x_1 \otimes y_1 & x_0 \ar[r,swap,"x"] & x_1 & y_1 \ar[r,equal] & y_1 \\
\end{tikzcd}
\]

The projections $\pi_X, \pi_Y$ induce a natural map $X \otimes Y \to X \times Y$, sending a cube $x \otimes y$ as above to the pair $(x \sigma^L_n, y \sigma^F_m)$. It is this natural map which will be our key tool in comparing the geometric and cartesian products. 

For $B \subseteq \FS$ with $A \subseteq B$ and $\Sigma, \delta \in B$, we also define a natural map $X \times Y \to i^* i_! (X \otimes Y)$, as follows: given a pair of cubes $x \colon \Box^n_A \to X, y \colon \Box^n_A \to Y$, we send the pair $(x,y)$ to $(x \otimes y) D_n$. We now verify that this defines a valid map of cubical sets.

\begin{prop}\label{cartesian-geo-map-valid}
    For every $X, Y \in \cSet_A$, the assignment above defines a map of cubical sets $X \times Y \to i^* i_! (X \otimes Y)$.
\end{prop}

\begin{proof}
We must show that the given assignment is compatible with the structure maps of $X \times Y$. To that end, consider $x \colon \Box^n_A \to X, y \colon \Box^n_A \to Y$, and a map $\phi \colon \Box^m_A \to \Box^n_A$. Then, applying \cref{geo-prod-characterization,long-diagonal-duplicates}, we can see that $(x,y)\phi = (x\phi,y\phi)$ is mapped to:

\begin{align*}
    (x\phi \otimes y\phi) D_m & = (x \otimes y) (\phi \otimes \phi) D_m \\
    & = (x \otimes y) D_n \phi \\
\end{align*}

Thus the claim is proven.
\end{proof}

\begin{lem}\label{X-Y-composite-unit}
    The composite $X \otimes Y \to X \times Y \to i^* i_! (X \otimes Y)$ of the two maps defined above is equal to the unit of the adjunction $i_! \adjoint i^*$.
\end{lem}

\begin{proof}
    Given $x \colon \Box^m_A \to X$ and $y \colon \Box^n_A \to Y$, by \cref{geo-prod-characterization,long-diagonal-retraction} the image of $x \otimes y$ under the composite map is 
    \begin{align*}
        (x \sigma^L_n \otimes y \sigma^F_m)D_{m+n} & = (x \otimes y) (\sigma^L_n \otimes \sigma^F_m) D_{m+n} \\
        & = x \otimes y
    \end{align*}
 Thus the composite and the unit agree on all cubes of the form $x \otimes y$; by \cref{geo-prod-characterization}, this is sufficient to show that they agree on all cubes of $X \otimes Y$.
\end{proof}

This result provides us with a proof of the well-known fact that in categories of cubical sets with symmetries and diagonals, the geometric and cartesian products coincide.

\begin{prop}\label{geometric-cartesian-iso}
    If $A \subseteq \FS$ with $\Sigma, \delta \in A$, then the natural map $X \otimes Y \to X \times Y$ is an isomorphism.
\end{prop}

\begin{proof}
    In this case, we may apply \cref{X-Y-composite-unit} with $A = B$, so that $i_! \adjoint i^*$ is the identity adjunction on $\cSet_A$. Thus we see that the composite $X \otimes Y \to X \times Y \to X \otimes Y$ is the identity on $X \otimes Y$.

    It thus remains to consider the composite $X \times Y \to X \otimes Y \to X \times Y$. Given $x \colon \Box^n_A \to X, y \colon \Box^n_A \to Y$, the map $X \times Y \to X \otimes Y$ sends the pair $(x,y)$  to $(x \otimes y)D_n$. The map $X \otimes Y \to X \times Y$ then sends this cube to $(x \sigma^L_n , y \sigma^F_n) D_n$. By definition, this is equal to $(x \sigma^L_n D_n, y \sigma^F_n D_n)$; by \cref{long-diagonal-retraction}, this is equal to $(x,y)$. Thus this composite is the identity on $X \times Y$.
    \end{proof}
    
\subsection{Cartesian monoidality of cubical Joyal model structures}

Our next goal is to use the comparison map between the geometric and cartesian products to prove our desired cartesian monoidality results for the cubical Joyal model structures. Throughout this subsection, fix $A, B$ as in the previous subsection, with the additional assumptions that $A$ is a non-empty subset of $\conset$ and $\rho \notin B$. For concreteness, we will assume $\wedge \in A$; as in \cref{sec unit}, the case where $\vee \in A$ follows by parallel arguments involving negative connections.

\begin{lem}\label{X-Y-comparison-mono}
    For all $X, Y \in \cSet_A$, the morphisms $X \otimes Y \to X \times Y$ and $X \times Y \to i^* i_! (X \otimes Y)$ are monomorphisms.
\end{lem}

\begin{proof}
    For $X \otimes Y \to X \times Y$, this is immediate from \cref{unit-mono,X-Y-composite-unit}. 
    
    To see that $X \times Y \to i^* i_! (X \otimes Y)$ is a monomorphism, suppose that for some cubes $x, x' \colon \Box^n_A \to X, y, y' \colon \Box^n_A \to Y$, the pairs $(x,y)$ and $(x',y')$ are mapped to the same cube of $i^* i_! (X \otimes Y)$ -- in other words, that $(x \otimes y) D_n = (x' \otimes y') D_n$. Applying the map $i^* i_! \pi_X$, we see that $x \sigma^L_n D_n = x' \sigma^L_n D_n$ in $i^* i_! X$. By \cref{long-diagonal-retraction}, it follows that $x = x'$ as cubes of $i^* i_! X$; thus $x = x'$ in $X$ by \cref{unit-mono}. A similar proof shows that $y = y'$.
\end{proof}

In particular, considering the case where $X$ and $Y$ are representable, we have for any $m, n \geq 0$ a composable pair of monomorphisms
\[
\Box^{m+n}_A \to \Box^m_A \times \Box^n_A \to i^* \Box^{m+n}_B
\]

We aim to show that this composable pair of monomorphisms witnesses $\Box^m_A \times \Box^n_A$ as a decomposition-closed subcomplex of $i^* i_! \Box^{m+n}_B$, in order to apply the results of \cref{sec N-closed}. To do this, we will first characterize the image in $i^* \Box^{m+n}_B$ of $\Box^m_A \times \Box^n_A$.

\begin{prop} \label{cartesian-image}
    The image of the map $\Box^m_A \times \Box^n_A \to i^* \Box^{m+n}_B$ consists of all morphisms $\phi$ in $\Box_B$ with codomain $\Box^{m+n}_B$ such that $\sigma^F_m \phi$ and $\sigma^L_n \phi$ are in $\Box_A$.
\end{prop}

\begin{proof}
    We first show that any $k$-cube $\phi$ of $i^* \Box^{m+n}_B$ which satisfies the given criterion when considered as a morphism in $\Box_B$ is contained in the image of $\Box^m_A \times \Box^n_A$. In this case, the pair $(\sigma^L_n \phi, \sigma^F_m \phi)$ defines a $k$-cube of $\Box^m_A \times \Box^n_A$, which is mapped to $(\sigma^L_n \phi \otimes \sigma^F_m \phi)D_k$. Using \cref{geo-prod-characterization,long-diagonal-duplicates}, we may calculate:
    \begin{align*}
        (\sigma^L_n \phi \otimes \sigma^F_m \phi)D_k & = (\sigma^L_n \otimes \sigma^F_m) (\phi \otimes \phi) D_k \\
        & = (\sigma^L_n \otimes \sigma^F_m) D_{m+n} \phi \\
        & = \phi
    \end{align*}

    Next we will show that all cubes in the image of $\Box^m_A \times \Box^n_A$ satisfy the given criterion. Consider a pair of $k$-cubes $\phi \colon \Box^k_A \to \Box^m_A, \phi' \colon \Box^k_A \to \Box^n_A$. The $k$-cube $(\phi,\phi')$ of $\Box^m_A \times \Box^n_A$ is mapped to $(\phi \otimes \phi')D_k$. A straightforward calculation involving \cref{long-diagonal-duplicates,long-diagonal-retraction} shows that the following diagram commutes:
    \[
    \begin{tikzcd}
    \Box^k_B \ar[r,"D_k"] \ar[dr,equal] & \Box^k_B \otimes \Box^k_B \ar[r,"\phi \otimes \phi'"] \ar[d,"\sigma^F_k"] & \Box^m_B \otimes \Box^n_B \ar[d,"\sigma^F_m"] \\
    & \Box^k_B \ar[r,"\phi'"] & \Box^n_B \\
    \end{tikzcd}
    \]
    Thus $\sigma^F_m(\phi \otimes \phi')D_k = \phi'$. A similar proof shows $\sigma^L_n(\phi \otimes \phi')D_k = \phi$.
\end{proof}

\begin{prop} \label{cartesian-N-closed}
    For any $m, n \geq 0$, the maps $\Box^{m+n}_A \to \Box^m_A \times \Box^n_A \to i^* \Box^{m+n}_B$ witness $\Box^m_A \times \Box^n_A$ as a decomposition-closed subcomplex of $i^* \Box^{m+n}_B$.
\end{prop}

\begin{proof}
    The two maps are monomorphisms by \cref{X-Y-comparison-mono}, and compose to the unit of $i_! \adjoint i^*$ by \cref{X-Y-composite-unit}. Thus it remains to show that if a cube $\kappa (\phi \otimes \Box^k_B)$ of $i^* \Box^{m+n}_B$, where $\phi$ is active and $\kappa$ is a composite of face maps, is contained in the image of $\Box^m_A \otimes \Box^n_A$, then so is $\kappa N_k(\phi)$. By \cref{cartesian-image}, this is equivalent to showing that if $\sigma^F_m \kappa (\phi \otimes \Box^k_B)$ and $\sigma^L_n \kappa (\phi \otimes \Box^k_B)$ are contained in $\Box_A$, then so are $\sigma^F_m \kappa N_k(\phi)$ and $\sigma^L_n \kappa N_k(\phi)$.

    We first note that for any face composite $\kappa \colon \Box^p \to \Box^{m+n}$ there exist commuting diagrams as below for some face composites $\kappa' \colon \Box^{p'}_B \to \Box^m_B, \kappa'' \colon \Box^{p''}_B \to \Box^n_B$.
    \[
    \begin{tikzcd}
        \Box^p_B \ar[r,"\sigma^L_{p-p'}"] \ar[d,swap,"\kappa"] & \Box^{p'}_B \ar[d,"\kappa'"] & \Box^p_B \ar[r,"\sigma^F_{p-p''}"] \ar[d,swap,"\kappa"] & \Box^{p''}_B \ar[d,"\kappa''"] \\
        \Box^{m+n}_B \ar[r,"\sigma^L_n"] & \Box^m_B & \Box^{m+n}_B \ar[r,"\sigma^F_m"] & \Box^n_B\\
    \end{tikzcd}
    \]
    Thus the statement we aim to prove can be rephrased as: if $\kappa' \sigma^L_{p-p'} (\phi \otimes \Box^k_B)$ and $\kappa'' \sigma^F_{p-p''} (\phi \otimes \Box^k_B)$ are in $\Box_A$, then so are $\kappa' \sigma^L_{p-p'} N_k(\phi)$ and $\kappa'' \sigma^F_{p-p''} N_k(\phi)$.

    Moreover, for any face composite $\kappa \colon \Box^r_B \to \Box^s_B$ and any $\psi \colon \Box^q_B \to \Box^r_B$, the composite $\kappa \psi$ is in $\Box_A$ if and only if $\psi$ is in $\Box_A$ (because every face map is in $\Box_A$, and has a retraction in $\Box_A$.) Thus it will suffice to prove the following: for any morphism $\psi \colon \Box^q_B \to \Box^r_B$ and any $k \geq 0$ and $0 \leq i \leq r+k$, if $\sigma^F_i (\psi \otimes \Box^k_B)$ is in $\Box_A$, then so is $\sigma^F_i N_k(\psi)$, and likewise, if $\sigma^L_i (\psi \otimes \Box^k_B)$ is in $\Box_A$ then so is $\sigma^L_i N_k(\psi)$. Throughout our proof, a tuple $(a_1,\ldots,a_q,b,c_1,\ldots,c_k)$ in $\Box^{q+1+k}_B$ will be abbreviated by $(\vec{a},b,\vec{c})$.

    First we prove the statement involving post-composition with $\sigma^F_i$. If $i < r$, then $\sigma^F_i = \sigma^F_i \otimes \Box^k_B$, so that $\sigma^F_i (\psi \otimes \Box^k_B) = \sigma^F_i \psi \otimes \Box^k_B$. This map is in $\Box_A$ by assumption, therefore $\sigma^F_i\psi \otimes \Box^1_B \otimes \Box^k_B$ is in $\Box_A$ as well. Hence, the composite $\gamma_{r-i,1} (\sigma^F_i\psi \otimes \Box^1_B \otimes \Box^k_B)$ is also in $\Box_A$. 
    Using \cref{N-explicit}, we may compute that this composite sends a tuple $(\vec{a},b,\vec{c})$ to:
    \[
    (\psi(\vec{a})_{i+1},\ldots,\psi(\vec{a})_{r-1}, \psi(\vec{a})_r \wedge b, c_1,\ldots,c_k)
    \]
    We may likewise compute that this is the image of $(\vec{a},b,\vec{c})$ under $\sigma^F_i N_k(\psi)$.

    If $i \geq r$, then $\sigma^F_i N_k(\psi)$ sends $(\vec{a},b,\vec{c})$ to:
    \[
    (c_{i-r+1},\ldots,c_k)
    \]
    Thus this map is the degeneracy $\sigma^F_{q+1+i-r}$.

    We now consider the case of post-composition with $\sigma^L_i$. In this case, if $1 \leq i \leq k$, then we may note that $\sigma^L_i = \Box^r_B \otimes \sigma^L_i$. We thus have $\sigma^L_i (\psi \otimes \Box^k_B) = \psi \otimes \sigma^L_i = (\psi \otimes \Box^{k-i}_B) \sigma^L_i$.

    So $(\psi \otimes \Box^{k-i}_B) \sigma^L_i$ is in $\Box_A$; pre-composing with a section of $\sigma^L_i$, we observe that $\psi \otimes \Box^{k-i}_B$ is in $\Box_A$, hence so is $\psi \otimes \Box^{1+k}_B$. It follows that the composite $\sigma^L_i \gamma_{r,1} (\psi \otimes \Box^{1+k}_B) = \sigma^L_i N_k(\phi)$ is in $\Box_A$ as well. 

    If $k < i \leq r+k$, then $\sigma^L_i N_k(\psi)$ sends $(\vec{a},b,\vec{c})$ to:
    \[
    (\psi(\vec{a})_1,\ldots,\psi(\vec{a})_{r-(i-k)})
    \]
    Thus this map is equal to $\sigma^L_i (\psi \otimes \Box^k_B)$, and is therefore in $\Box_A$ by assumption.
\end{proof}

\begin{cor}\label{X-Y-comparison-weq}
    For all $X, Y \in \cSet_A$, the natural maps $X \otimes Y \to X \times Y$ and $X \times Y \to i^* i_! (X \otimes Y)$ are trivial cofibrations.
\end{cor}

\begin{proof}
    Both maps are monomorphisms by \cref{X-Y-comparison-mono}, so we need only prove that they are weak equivalences. Because the bifunctors $\otimes, \times$, and $i^* i_! \otimes$ all preserve cofibrations and colimits in each variable, by \cref{cube-EZ} and a standard induction on skeleta argument it suffices to consider the case where both $X$ and $Y$ are representable. This case is immediate from \cref{N-closed-anodyne,cartesian-N-closed}
\end{proof}

\begin{cor} \label{cartesian-pres-weq}
    The bifunctor $\times \colon \cSet_A \to \cSet_A$ preserves weak equivalences in each variable.
\end{cor}

\begin{proof}
    This is immediate from \cref{geometric-monoidal-models,X-Y-comparison-weq}.
\end{proof}

\begin{thm}\label{cartesian-monoidal}
    The cubical Joyal model structure on $\cSet_A$ is cartesian monoidal.
\end{thm}

\begin{proof}
    Since all objects are cofibrant, we need only prove that the cartesian product on $\cSet_A$ satisfies the pushout product axiom. That a pushout product of cofibrations is a cofibration follows from the corresponding result for monomorphisms of sets. Thus the pushout product axiom follows from \cref{pop-axiom-sufficient,cartesian-pres-weq}.
\end{proof}

We likewise obtain a new proof of the following previously known result.

\begin{thm}\label{Grothendieck-cartesian-monoidal}
    The Grothendieck model structure on $\cSet_A$ is cartesian monoidal.
\end{thm}

\begin{proof}
    The natural map $X \otimes Y \to X \times Y$ is a trivial cofibration in the Grothendieck model structure by \cref{cubical-Joyal-localization,X-Y-comparison-weq}, so we may proceed exactly as in the proof of \cref{cartesian-pres-weq,cartesian-monoidal}.
\end{proof}

\bibliographystyle{amsalphaurlmod}
\bibliography{general-bibliography}

\end{document}